\documentclass[11pt,a4paper]{article}
\usepackage{amsfonts}
\usepackage{amssymb}
\usepackage{amsmath}
\usepackage{graphicx}
\newtheorem{theorem}{Theorem}[section]
\newtheorem{corollary}{Corollary}[section]
\newtheorem{definition}{Definition}[section]
\newtheorem{remark}{Remark}[section]

\newenvironment{proof}[1][Proof]{\textbf{#1.} }{\ \rule{0.5em}{0.5em}}

\begin{document}

\title{\bigskip A generalization of the Minkowski distance\\
and a new definition of the ellipse}
\author{Harun Bar\i \c{s} \c{C}olako\u{g}lu\bigskip \\
{\small hbcolakoglu@akdeniz.edu.tr}\medskip \\
{\small Akdeniz University, Vocational School of Technical Sciences,}\\
{\small Department of Computer Technologies, 07070, Konyaalt\i , Antalya, T\"{U}RK\.{I}YE.}}
\date{}
\maketitle

\begin{abstract}
\noindent In this paper, we generalize the Minkowski distance by defining a
new distance function in \textit{n}-dimensional space, and we show that this
function determines also a metric family as the Minkowski distance. Then, we
consider three special cases of this family, which generalize the taxicab,
Euclidean and maximum metrics respectively, and finally we determine circles
of them with their some properties in the real plane. While we determine
some properties of circles of the generalized Minkowski distance, we also
discover a new definition for the ellipse.
\end{abstract}
\title{\noindent \textbf{Keywords: }Minkowski distance, $l_{p}$-metric,
taxicab distance, Manhattan distance, Euclidean distance, maximum distance,
Chebyshev distance, circle, ellipse, conjugate diameter, eccentrix.\medskip \newline
\textbf{2010 MSC: }51K05, 51K99, 51N20.\medskip }

\section{Introduction}

\noindent Beyond the mathematics; distances, especially the well-known
Minkowski distance (also known as $l_{p}$-metric) with its special cases
taxicab (also known as $l_{1}$ or Manhattan), Euclidean (also known as $%
l_{2} $) and maximum (also known as $l_{\infty }$ or Chebyshev) distances,
are very important keys for many application areas such as data mining,
machine learning, pattern recognition and spatial analysis 
(see \cite{Amorim}, \cite{Cha}, \cite{Doherty}, \cite{Li}, \cite{Lu}, \cite{Shirkhorshidi}, 
\cite{Singh} and \cite{Vezzetti} for some of related studies).

\noindent Here, we generalize the Minkowski distance for \textit{n}%
-dimensional case, and we show that this generalization gives a new metric
family for $p\geq 1$ as the Minkowski distance itself. Then, we give some
basic distance properties of this generalized Minkowski distance, and we
consider the new metric family for cases $p=1$, $p=2$ and $p\rightarrow
\infty $, which we call the generalized taxicab, Euclidean and maximum
metrics respectively. Finally, we determine circles of them in the real
plane. We see that circles of the generalized taxicab and maximum metrics
are parallelograms and circles of the generalized Euclidean metric are
ellipses. While we determine some properties of circles of the generalized
Euclidean distance, we also discover a new definition for the ellipse, which
can be referenced by "two-eccentrices" definition, as the well-known
"two-foci" and "focus-directrix" definitions.

\noindent Throughout this paper, symmetry about a line is used in the
Euclidean sense and angle measurement is in\ Euclidean radian. Also the
terms square, rectangle, rhombus, parallelogram and ellipse are used in the
Euclidean sense, and center of them stands for their center of symmetry.

\section{A generalization of the Minkowski distance\protect\vspace{-0.09in}}

\noindent We generalize the Minkowski distance using linearly independent $n$
unit vectors $v_{1},...,v_{n}$ and $n$ positive real numbers $\lambda
_{1},...,\lambda _{n}$, as in the following definition. For the sake of
shortness we use notation $d_{p(v_{1},...,v_{n})}$, instead of for example $%
d_{p(v_{1},...,v_{n})}^{(\lambda _{1},...,\lambda _{n})}$, for the new
distance family, supposing $\lambda _{i}$ weights are initially determined
and fixed, and we call it $(v_{1},...,v_{n})$-\emph{Minkowski distance}.

\begin{definition}
Let $X=(x_{1},...,x_{n})$ and $Y=(y_{1},...,y_{n})$ be two points in $%
\mathbb{R}^{n}$. For linearly independent $n$ unit vectors $v_{1},...,v_{n}$
where $v_{i}=(v_{i1},...,v_{in})$, and positive real numbers $p,\lambda
_{1},...,\lambda _{n}$, the function $d_{p(v_{1},...,v_{n})}:\mathbb{R}%
^{n}\times \mathbb{R}^{n}\rightarrow \lbrack 0,\infty )$ defined by\vspace{%
-0.12in}%
\begin{equation}
d_{p(v_{1},...,v_{n})}(X,Y)=\left( \sum_{i=1}^{n}\left( \lambda
_{i}\left\vert v_{i1}(x_{1}-y_{1})+...+v_{in}(x_{n}-y_{n})\right\vert
\right) ^{p}\right) ^{1/p}\vspace{-0.06in}
\end{equation}%
is called $\boldsymbol{(}v_{1},...,v_{n}\boldsymbol{)}$\textbf{-Minkowski }%
(or\textbf{\ }$l_{p(v_{1},...,v_{n})}$)\textbf{\ distance\ function} in $%
\mathbb{R}^{n}$, and real number $d_{p(v_{1},...,v_{n})}(X,Y)$ is called $%
\boldsymbol{(}v_{1},...,v_{n}\boldsymbol{)}$\textbf{-Minkowski distance}%
\textit{\ }between points $X$ and $Y$. In addition, if $p=1$, $p=2$ and $%
p\rightarrow \infty $, then $d_{p(v_{1},...,v_{n})}(X,Y)$ is called $%
\boldsymbol{(}v_{1},...,v_{n}\boldsymbol{)}$\textbf{-taxicab distance}%
\textit{, }$\boldsymbol{(}v_{1},...,v_{n}\boldsymbol{)}$\textbf{-Euclidean
distance }\textit{and }$\boldsymbol{(}v_{1},...,v_{n}\boldsymbol{)}$\textbf{%
-maximum distance }between points $X$ and $Y$ respectively, and we denote
them by $d_{T(v_{1},...,v_{n})}(X,Y)$, $d_{E(v_{1},...,v_{n})}(X,Y)$ and $%
d_{M(v_{1},...,v_{n})}(X,Y)$ respectively.\vspace{-0.07in}
\end{definition}

\noindent Here, since$\ \sigma \hspace{-0.03in}\leq \hspace{-0.03in}%
d_{p(v_{1},...,v_{n})}(X,Y)\hspace{-0.03in}\leq \hspace{-0.03in}\sigma
n^{1/p}$ where $\sigma \hspace{-0.03in}=\hspace{-0.03in}\max\limits_{i\in
\{1,...,n\}}\hspace{-0.03in}\left\{ \lambda _{i}\left\vert v_{i1}(x_{1}%
\hspace{-0.03in}-\hspace{-0.03in}y_{1})\hspace{-0.03in}+\hspace{-0.03in}...%
\hspace{-0.03in}+\hspace{-0.03in}v_{in}(x_{n}\hspace{-0.03in}-\hspace{-0.03in%
}y_{n})\right\vert \right\} $, we have that $\lim\limits_{p\rightarrow
\infty }d_{p(v_{1},...,v_{n})}(X,Y)\hspace{-0.03in}=\hspace{-0.08in}%
\max\limits_{i\in \{1,...,n\}}\hspace{-0.03in}\left\{ \lambda _{i}\left\vert
v_{i1}(x_{1}\hspace{-0.03in}-\hspace{-0.03in}y_{1})\hspace{-0.03in}+\hspace{%
-0.03in}...\hspace{-0.03in}+\hspace{-0.03in}v_{in}(x_{n}\hspace{-0.03in}-%
\hspace{-0.03in}y_{n})\right\vert \right\} $ and so\vspace{-0.07in}%
\begin{equation}
d_{M(v_{1},...,v_{n})}(X,Y)=\max\limits_{i\in \{1,...,n\}}\hspace{-0.03in}%
\left\{ \lambda _{i}\left\vert v_{i1}(x_{1}\hspace{-0.03in}-\hspace{-0.03in}%
y_{1})\hspace{-0.03in}+\hspace{-0.03in}...\hspace{-0.03in}+\hspace{-0.03in}%
v_{in}(x_{n}\hspace{-0.03in}-\hspace{-0.03in}y_{n})\right\vert \right\} .
\end{equation}

\begin{remark}
In $n$-dimensional Cartesian coordinate space, let $\Psi _{P}^{v_{i}}$
denote hyperplane through point $P$ and perpendicular to the vector $v_{i}$
for $i\in \{1,...,n\}$. Since Euclidean distance between the point $Y$ and
hyperplane $\Psi _{X}^{v_{i}}$ (or the point $X$ and hyperplane $\Psi
_{Y}^{v_{i}}$) is 
\begin{equation}
d_{E}(Y,\Psi _{X}^{v_{i}})=\left\vert
v_{i1}(x_{1}-y_{1})+...+v_{in}(x_{n}-y_{n})\right\vert ,
\end{equation}
$(v_{1},...,v_{n})$-Minkowski distance between the points $X$ and $Y$ is
\vspace{-0.08in}
\begin{equation}
d_{p(v_{1},...,v_{n})}(X,Y)=\left( \sum_{i=1}^{n}\text{\thinspace }(\lambda
_{i}d_{E}(Y,\Psi _{X}^{v_{i}}))^{p}\right) ^{1/p}
\end{equation}
which is the geometric interpretation of $(v_{1},...,v_{n})$-Minkowski distance.
In other words, $(v_{1},...,v_{n})$-Minkowski distance between points $X$
and $Y$, is determined by the sum of weighted Euclidean distances from one
of the points to hyperplanes through the other point, each of which is
perpendicular to one of the vectors $v_{1},...,v_{n}$. Clearly, for $\lambda
_{i}=1$ and unit vectors $v_{1},...,v_{n}$ where $v_{ii}=1$ and $v_{ij}=0$
for $i\neq j$, $i,j\in \{1,...,n\}$, we have
\begin{equation}
d_{p(v_{1},...,v_{n})}(X,Y)=d_{p}(X,Y)=\left( \sum_{i=1}^{n}\left\vert
x_{i}-y_{i}\right\vert ^{p}\right) ^{1/p}
\end{equation}
which is the well-known \textbf{Minkowski }(or\textbf{\ }$l_{p}$) \textbf{distance} 
between the points $X$ and $Y$, that gives the well-known 
\textbf{taxicab}, \textbf{Euclidean} and \textbf{maximum} distances denoted by 
$d_{T}(X,Y)$, $d_{E}(X,Y)$ and $d_{M}(X,Y)$, for $p=1$, $p=2$ and $%
p\rightarrow \infty $ respectively (see \cite[pp. 94, 301]{Deza}; see also 
\cite{Krause} and \cite{Salihova}).
\end{remark}

\noindent The following proposition shows that
$(v_{1},...,v_{n})$-Minkowski distance function for $p\geq 1$ satisfies the metric properties
in $\mathbb{R}^{n}$:

\begin{theorem}
For $p\geq 1$, $(v_{1},...,v_{n})$-Minkowski distance function determines
metric in $\mathbb{R}^{n}$.
\end{theorem}

\begin{proof}
Let $X=(x_{1},...,x_{n})$, $Y=(y_{1},...,y_{n})$ and $Z=(z_{1},...,z_{n})$
be three points in $\mathbb{R}^{n}$.

\noindent (M1) Clearly, if $X=Y$, then $d_{p(v_{1},...,v_{n})}(X,Y)=0$.
Conversely, if $d_{p(v_{1},...,v_{n})}(X,Y)=0$, then we get $A\mathbf{x}=0$
where 
\begin{equation*}
A=\left[ 
\begin{array}{ccc}
v_{11} & \ldots & v_{1n} \\ 
\vdots & \ddots & \vdots \\ 
v_{n1} & \ldots & v_{nn}%
\end{array}%
\right] \text{ \ and \ }\mathbf{x}=\left[ 
\begin{array}{c}
x_{1}-y_{1} \\ 
\vdots \\ 
x_{n}-y_{n}%
\end{array}%
\right] .
\end{equation*}%
Since $v_{1},...,v_{n}$ are linearly independent, we have $\left\vert
A\right\vert \neq 0$. Therefore the homogeneous system $A\mathbf{x}=0$ has
only trivial solution. Thus, we have $x_{i}-y_{i}=0$, and so $X=Y$.\smallskip

\noindent (M2) It is clear that $d_{p(v_{1},...,v_{n})}(X,Y)=d_{p(v_{1},...,v_{n})}(Y,X)$.\smallskip

\noindent (M3) The triangle inequality can be proven using the
Minkowski inequality for $p\geq 1$ (see \cite[p. 25]{Beck}) as follows:
\newline
$d_{p(v_{1},...,v_{n})}(X,Y)=\left( \sum\limits_{i=1}^{n}\left( \lambda
_{i}\left\vert v_{i1}(x_{1}\hspace{-0.03in}-\hspace{-0.03in}%
y_{1})+...+v_{in}(x_{n}\hspace{-0.03in}-\hspace{-0.03in}y_{n})\right\vert
\right) ^{p}\right) ^{1/p}$\newline
$=\left( \sum\limits_{i=1}^{n}\left\vert \lambda _{i}(v_{i1}(x_{1}\hspace{%
-0.03in}-\hspace{-0.03in}z_{1}\hspace{-0.03in}+\hspace{-0.03in}z_{1}\hspace{%
-0.03in}-\hspace{-0.03in}y_{1})+...+v_{in}(x_{n}\hspace{-0.03in}-\hspace{%
-0.03in}z_{n}\hspace{-0.03in}+\hspace{-0.03in}z_{n}\hspace{-0.03in}-\hspace{%
-0.03in}y_{n}))\right\vert ^{p}\right) ^{1/p}$\newline
$=\left( \sum\limits_{i=1}^{n}\left\vert \lambda _{i}(v_{i1}(x_{1}\hspace{%
-0.03in}-\hspace{-0.03in}z_{1})+...+v_{in}(x_{n}\hspace{-0.03in}-\hspace{%
-0.03in}z_{n}))\hspace{-0.03in}+\hspace{-0.03in}\lambda _{i}(v_{i1}(z_{1}%
\hspace{-0.03in}-\hspace{-0.03in}y_{1})+...+v_{in}(z_{n}\hspace{-0.03in}-%
\hspace{-0.03in}y_{n}))\right\vert ^{p}\right) ^{1/p}$\newline
$\leq \left( \sum\limits_{i=1}^{n}\left\vert \lambda _{i}(v_{i1}(x_{1}%
\hspace{-0.03in}-\hspace{-0.03in}z_{1})+...+v_{in}(x_{n}\hspace{-0.03in}-%
\hspace{-0.03in}z_{n}))\right\vert ^{p}\right) ^{1/p}\hspace{-0.03in}+%
\hspace{-0.03in}\left( \sum\limits_{i=1}^{n}\left\vert \lambda
_{i}(v_{i1}(z_{1}\hspace{-0.03in}-\hspace{-0.03in}y_{1})+...+v_{in}(z_{n}%
\hspace{-0.03in}-\hspace{-0.03in}y_{n}))\right\vert ^{p}\right) ^{1/p}$%
\newline
$=\left( \sum\limits_{i=1}^{n}\left( \lambda _{i}\left\vert v_{i1}(x_{1}%
\hspace{-0.03in}-\hspace{-0.03in}z_{1})+...+v_{in}(x_{n}\hspace{-0.03in}-%
\hspace{-0.03in}z_{n})\right\vert \right) ^{p}\right) ^{1/p}\hspace{-0.03in}+%
\hspace{-0.03in}\left( \sum\limits_{i=1}^{n}\left( \lambda _{i}\left\vert
v_{i1}(z_{1}\hspace{-0.03in}-\hspace{-0.03in}y_{1})+...+v_{in}(z_{n}\hspace{%
-0.03in}-\hspace{-0.03in}y_{n})\right\vert \right) ^{p}\right) ^{1/p}$%
\newline
$=d_{p(v_{1},...,v_{n})}(X,Z)+d_{p(v_{1},...,v_{n})}(Z,Y).$\bigskip
\end{proof}

\noindent Since the Minkowski inequality does not hold for $0<p<1$ 
(see \cite[pp. 26-27]{Beck}), $(v_{1},...,v_{n})$-Minkowski distance function family
does not hold the triangle inequality for $0<p<1$. So, for $0<p<1$, it does
not determine a metric in $\mathbb{R}^{n}$ while it determines a distance.
We denote by $\mathbb{R}_{p(v_{1},v_{2})}^{2}$ the real plane $\mathbb{R}^{2}$ equipped with the 
$(v_{1},v_{2})$-Minkowski metric.\smallskip

\noindent The following theorem shows that $(v_{1},...,v_{n})$-Minkowski
distance and Euclidean distance between two points on any given line $l$ are
directly proportional:

\begin{theorem}
Given two points $X$ and $Y$ on a line $l$ with the direction vector 
$u=(u_{1},...,u_{n})$. Then, 
\begin{equation}
d_{p(v_{1},...,v_{n})}(X,Y)=\phi _{p(v_{1},...,v_{n})}(l)d_{E}(X,Y)
\end{equation}%
where $\phi _{p(v_{1},...,v_{n})}(l)=\frac{\left(
\sum\limits_{i=1}^{n}\left( \lambda _{i}\left\vert
v_{i1}u_{1}+...+v_{in}u_{n}\right\vert \right) ^{p}\right) ^{1/p}}{\sqrt{u_{1}^{2}+...+u_{n}^{2}}}.$
\end{theorem}

\begin{proof}
For any two points $X=(x_{1},...,x_{n})$ and $Y=(y_{1},...,y_{n})$ there is 
$k\in \mathbb{R}$ such that $(x_{1}-y_{1},...,x_{n}-y_{n})=k(u_{1},...,u_{n})$. Then, we
have $d_{E}(X,Y)=\left\vert k\right\vert \sqrt{u_{1}^{2}+...+u_{n}^{2}}$ and
\newline
$d_{p(v_{1},...,v_{n})}(X,Y)=\left\vert k\right\vert \left(
\sum\limits_{i=1}^{n}\left( \lambda _{i}\left\vert
v_{i1}u_{1}+...+v_{in}u_{n}\right\vert \right) ^{p}\right) ^{1/p}$ which
complete the proof.\medskip
\end{proof}

\noindent Now, the following corollaries are trivial:

\begin{corollary}
If $W,$ $X$ and $Y,Z$ are pair of distinct points such that the lines
determined by them are the same or parallel, then 
\begin{equation}
d_{p(v_{1},...,v_{n})}(W,X)/d_{p(v_{1},...,v_{n})}(Y,Z)=d_{E}(W,X)/d_{E}(Y,Z).
\end{equation}
\end{corollary}

\begin{corollary}
Circles and spheres of $(v_{1},...,v_{n})$-Minkowski distance are symmetric
about their center.
\end{corollary}

\begin{corollary}
Translation by any vector preserves $(v_{1},...,v_{n})$-Minkowski distance.
\end{corollary}

\begin{corollary}
For a vector $x=(x_{1},...,x_{n})$ in $\mathbb{R}_{p(v_{1},...,v_{n})}^{n}$, the induced norm is 
\begin{equation}
\left\Vert x\right\Vert _{p(v_{1},...,v_{n})}=\left( \sum_{i=1}^{n}\left(
\lambda _{i}\left\vert v_{i1}x_{1}+...+v_{in}x_{n}\right\vert \right)
^{p}\right) ^{1/p}.
\end{equation}
\end{corollary}

\begin{remark}
Instead of unit vectors, one can define $(v_{1},...,v_{n})$-Minkowski
distance for any linearly independent $n$ vectors $v_{1},...,v_{n}$ as
follows
\begin{equation}
d_{p(v_{1},...,v_{n})}^{\prime }(X,Y)=\left( \sum_{i=1}^{n}\left( \lambda
_{i}\frac{\left\vert v_{i1}(x_{1}-y_{1})+...+v_{in}(x_{n}-y_{n})\right\vert 
}{\left( v_{i1}^{2}+...+v_{in}^{2}\right) ^{1/2}}\right) ^{p}\right) ^{1/p}
\end{equation}%
or one can define it by unit vectors $v_{1},...,v_{n}$, and positive real
numbers $\mu _{1},...,\mu _{n}$ as follows
\begin{equation}
d_{p(v_{1},...,v_{n})}^{\prime \prime }(X,Y)=\left( \sum_{i=1}^{n}\mu
_{i}\left( \left\vert v_{i1}(x_{1}-y_{1})+...+v_{in}(x_{n}-y_{n})\right\vert
\right) ^{p}\right) ^{1/p}.
\end{equation}
These distance functions also determines metric families for $p\geq 1$,
generalizing the Minkowski distance. But then, we have 
\begin{equation}
d_{p(v_{1},...,v_{n})}^{\prime
}(X,Y)=d_{p(k_{1}v_{1},...,k_{n}v_{n})}^{\prime }(X,Y)\text{ for any }k_{i}\in \mathbb{R}-\{0\},
\end{equation}
and 
\begin{equation}
d_{M(v_{1},...,v_{n})}^{\prime \prime }(X,Y)=\lim\limits_{p\rightarrow
\infty }d_{p(v_{1},...,v_{n})}(X,Y)\hspace{-0.03in}=\hspace{-0.03in}%
\max\limits_{i\in \{1,...,n\}}\hspace{-0.03in}\left\{ \left\vert v_{i1}(x_{1}%
\hspace{-0.03in}-\hspace{-0.03in}y_{1})\hspace{-0.03in}+\hspace{-0.03in}...%
\hspace{-0.03in}+\hspace{-0.03in}v_{in}(x_{n}\hspace{-0.03in}-\hspace{-0.03in%
}y_{n})\right\vert \right\} ,
\end{equation}%
which is independent from $\mu _{i}$. However, we see that for every $p$
values, circles of $d_{p(v_{1},v_{2})}$ distance having the same center and
radius, have four common points, and they are nested inside one another. So,
it is easier to illustrate their circles in a figure (see Figure 1 for some
examples of $(v_{1},v_{2})$-Minkowski circles having the same center and
radius).
\end{remark}

\begin{center}
\includegraphics[width=4 in]{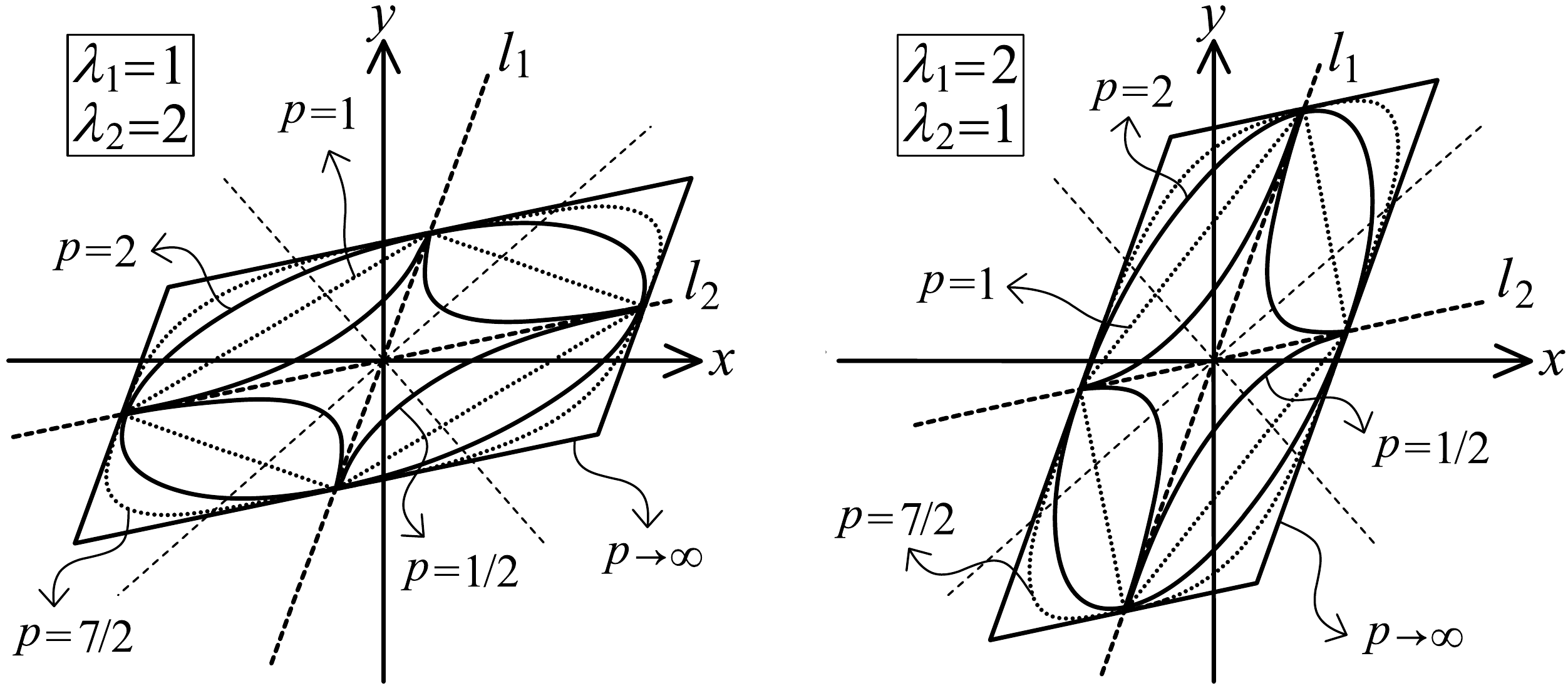}
\end{center}
\begin{center}
\textbf{Figure 1.} The unit $(v_{1},v_{2})$-Minkowski circles; $v_{1}=\left( \frac{3}{\sqrt{10}},\frac{-1}{\sqrt{10}}\right) $, $v_{2}=\left( \frac{-1}{\sqrt{26}},\frac{5}{\sqrt{26}}\right) $.
\end{center}

\noindent In the next sections, we investigate circles of $\mathbb{R}_{p(v_{1},v_{2})}^{2}$ 
for $p=1$, $p=2$ and $p\rightarrow \infty $, that we
call them $(v_{1},v_{2})$-taxicab, $(v_{1},v_{2})$-Euclidean and $(v_{1},v_{2})$-maximum circles respectively, having the case of $p=2$ at the
last in which we use circles of the other two cases. We use the coordinate
axes $x$ and $y$ as usual, instead of $x_{1}$ and $x_{2}$, while we
investigate circles of $\mathbb{R}_{p(v_{1},v_{2})}^{2}$, and throughout the paper, 
we denote by $l_{1}$ and $l_{2}$, the lines through center $C$ of a $(v_{1},v_{2})$-Minkowski circle
and perpendicular to unit vectors $v_{1}$ and $v_{2}$ respectively, that is $l_{i}=\Psi _{C}^{v_{i}}$.

\section{Circles of the generalized taxicab metric in $\mathbb{R}^{2}$}

\noindent By Definition 2.1 and Remark 2.1, $(v_{1},v_{2})$-taxicab distance
between points $P_{1}=(x_{1},y_{1})$ and $P_{2}=(x_{2},y_{2})$ in $\mathbb{R}^{2}$ is
\begin{eqnarray*}
d_{T(v_{1},v_{2})}(P_{1},P_{2}) &=&\lambda _{1}\left\vert
v_{11}(x_{1}-x_{2})+v_{12}(y_{1}-y_{2})\right\vert +\lambda _{2}\left\vert
v_{21}(x_{1}-x_{2})+v_{22}(y_{1}-y_{2})\right\vert \\
&=&\lambda _{1}\text{\thinspace }d_{E}(P_{2},l_{1})+\lambda _{2}\text{%
\thinspace }d_{E}(P_{2},l_{2})
\end{eqnarray*}%
that is, the sum of weighted Euclidean distances from the point $P_{2}$ to
the lines $l_{1}$ and $l_{2}$, which are passing through$\ P_{1}$ and
perpendicular to the vectors $v_{1}$ and $v_{2}$ respectively. For vectors 
$v_{1}=(1,0)$ and $v_{2}=(0,1)$, $d_{T(v_{1},v_{2})}$ in $\mathbb{R}^{2}$ is
the same as the (slightly) generalized taxicab metric (also known as the
weighted taxicab metric) defined in \cite{Wallen} (see also \cite{Col3} and 
\cite{Col4}). In addition, for unit vectors $v_{1}$ and $v_{2}$ such that 
$v_{1}\perp v_{2}$ and $v_{12}/v_{11}=m$ where $v_{11}\neq 0$, 
$d_{T(v_{1},v_{2})}$ in $\mathbb{R}^{2}$ is the same as the $m$-generalized
taxicab metric $d_{T_{g}(m)}$ defined in \cite{Col1}.\medskip

\noindent The following theorem determines circles of the generalized
taxicab metric $d_{T(v_{1},v_{2})}$ in $\mathbb{R}^{2}$:

\begin{theorem}
Every $(v_{1},v_{2})$-taxicab circle is a parallelogram with the same
center, each of whose diagonals is perpendicular to $v_{1}$ or $v_{2}$. In
addition, if $\lambda _{1}=\lambda _{2}$ then it is a rectangle, if 
$v_{1}\perp v_{2}$ then it is a rhombus, and if $\lambda _{1}=\lambda _{2}$
and $v_{1}\perp v_{2}$ then it is a square.
\end{theorem}

\begin{proof}
Without loss of generality, let us consider the unit $(v_{1},v_{2})$-taxicab
circle. Clearly, it is the set of points $P=(x,y)$ in $\mathbb{R}^{2}$
satisfying the equation 
\begin{equation}
d_{T(v_{1},v_{2})}(O,P)=\lambda _{1}\,d_{E}(P,l_{1})+\lambda
_{2}\,d_{E}(P,l_{2})=1
\end{equation}%
where $l_{i}:v_{i1}x+v_{i2}y=0$ for $i=1,2$, that is
\begin{equation}
\lambda _{1}\left\vert v_{11}x+v_{12}y\right\vert +\lambda _{2}\left\vert
v_{21}x+v_{22}y\right\vert =1.
\end{equation}%
One can see that this equation is the image of $\left\vert x\right\vert
+\left\vert y\right\vert =1$ which is the well-known taxicab circle, under
the linear transformation 
\begin{equation}
T\left( \left[ 
\begin{array}{c}
x \\ 
y%
\end{array}%
\right] \right) =\left[ 
\begin{array}{cc}
\frac{v_{22}}{\lambda _{1}\tau } & \frac{-v_{12}}{\lambda _{2}\tau } \\ 
\frac{-v_{21}}{\lambda _{1}\tau } & \frac{v_{11}}{\lambda _{2}\tau }%
\end{array}%
\right] \left[ 
\begin{array}{c}
x \\ 
y%
\end{array}%
\right]
\end{equation}
where $\tau =\left\vert 
\begin{array}{cc}
v_{11} & v_{12} \\ 
v_{21} & v_{22}%
\end{array}%
\right\vert $. Thus, the unit $(v_{1},v_{2})$-taxicab circle is a
parallelogram symmetric about the origin, having vertices $A_{1}=\left( 
\frac{v_{22}}{\lambda _{1}\tau },\frac{-v_{21}}{\lambda _{1}\tau }\right) $, 
$A_{2}=\left( \frac{-v_{12}}{\lambda _{2}\tau },\frac{v_{11}}{\lambda
_{2}\tau }\right) $, $A_{3}=\left( \frac{-v_{22}}{\lambda _{1}\tau },\frac{%
v_{21}}{\lambda _{1}\tau }\right) $, $A_{4}=\left( \frac{v_{12}}{\lambda
_{2}\tau },\frac{-v_{11}}{\lambda _{2}\tau }\right) $, and having diagonals
on the lines $l_{1}$ and $l_{2}$, each of which is perpendicular to $v_{1}$
or $v_{2}$, since 
\begin{equation*}
A_{1}A_{2}//A_{3}A_{4}\text{, }A_{1}A_{4}//A_{2}A_{3}\text{, }%
A_{2}A_{4}=l_{1}\text{ and }A_{1}A_{3}=l_{2}.
\end{equation*}%
In addition, if $\lambda _{1}=\lambda _{2}$ then 
$d_{E}(O,A_{1})=d_{E}(O,A_{2})$ and since a parallelogram having diagonals of
the same length is a rectangle, the unit $(v_{1},v_{2})$-taxicab circle is a
rectangle. Notice that sides of the rectangle are parallel to angle
bisectors of the lines $l_{1}$ and $l_{2}$. If $v_{1}\perp v_{2}$ then 
$l_{1}\perp l_{2}$ and since a parallelogram having perpendicular diagonals
is a rhombus, the unit $(v_{1},v_{2})$-taxicab circle is a rhombus. Finally,
it is clear that if $\lambda _{1}=\lambda _{2}$ and $v_{1}\perp v_{2}$, then
the unit $(v_{1},v_{2})$-taxicab circle is a square (see Figure 2 and Figure
3 for examples of the unit $(v_{1},v_{2})$-taxicab circles).
\end{proof}

\begin{center}
\includegraphics[width=4 in]{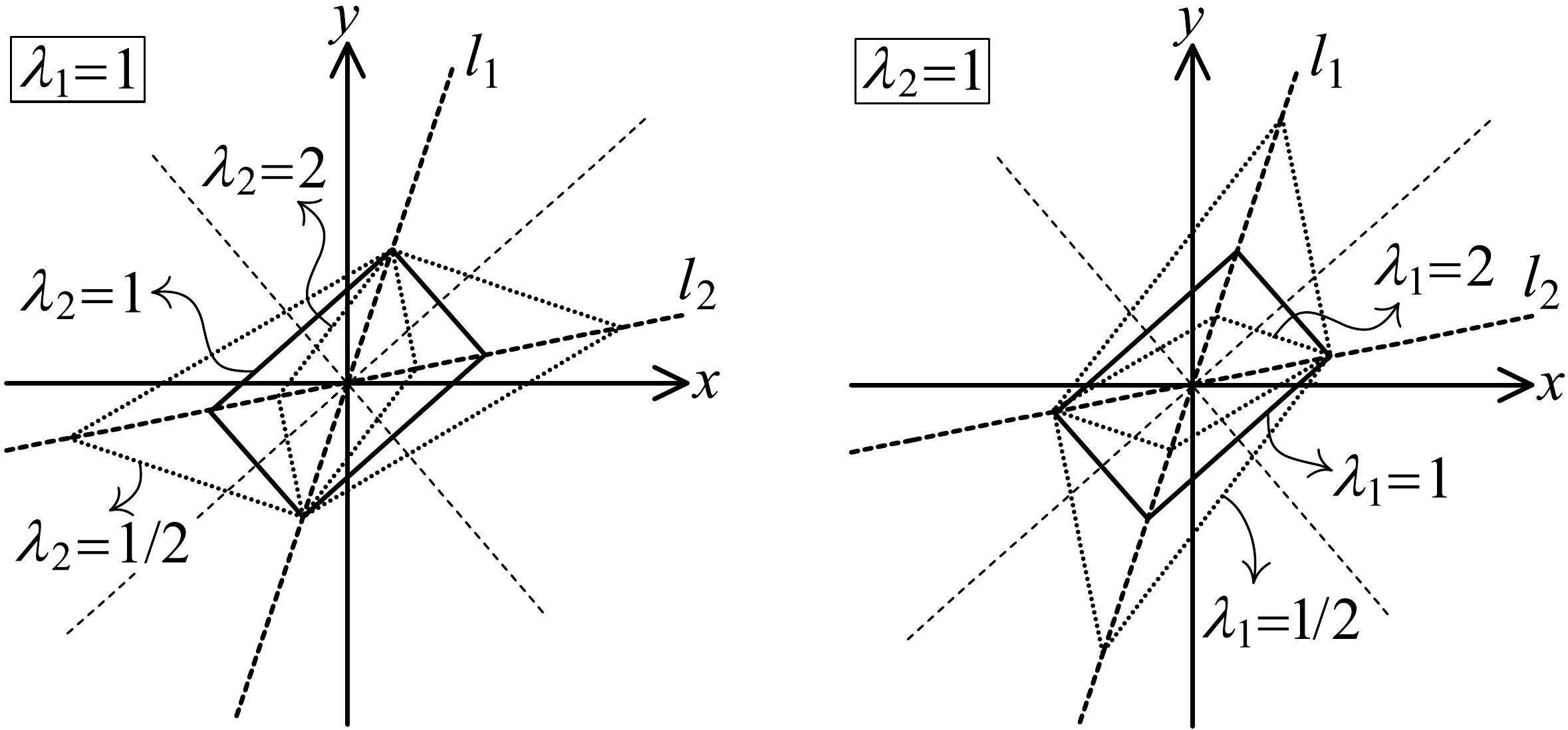} \vspace{-0.05in}
\end{center}
\begin{center}
\textbf{Figure 2.} The unit $(v_{1},v_{2})$-taxicab circles for 
$v_{1}=\left( \frac{3}{\sqrt{10}},\frac{-1}{\sqrt{10}}\right) $, 
$v_{2}=\left( \frac{-1}{\sqrt{26}},\frac{5}{\sqrt{26}}\right) $.
\end{center}

\begin{center}
\includegraphics[width=4 in]{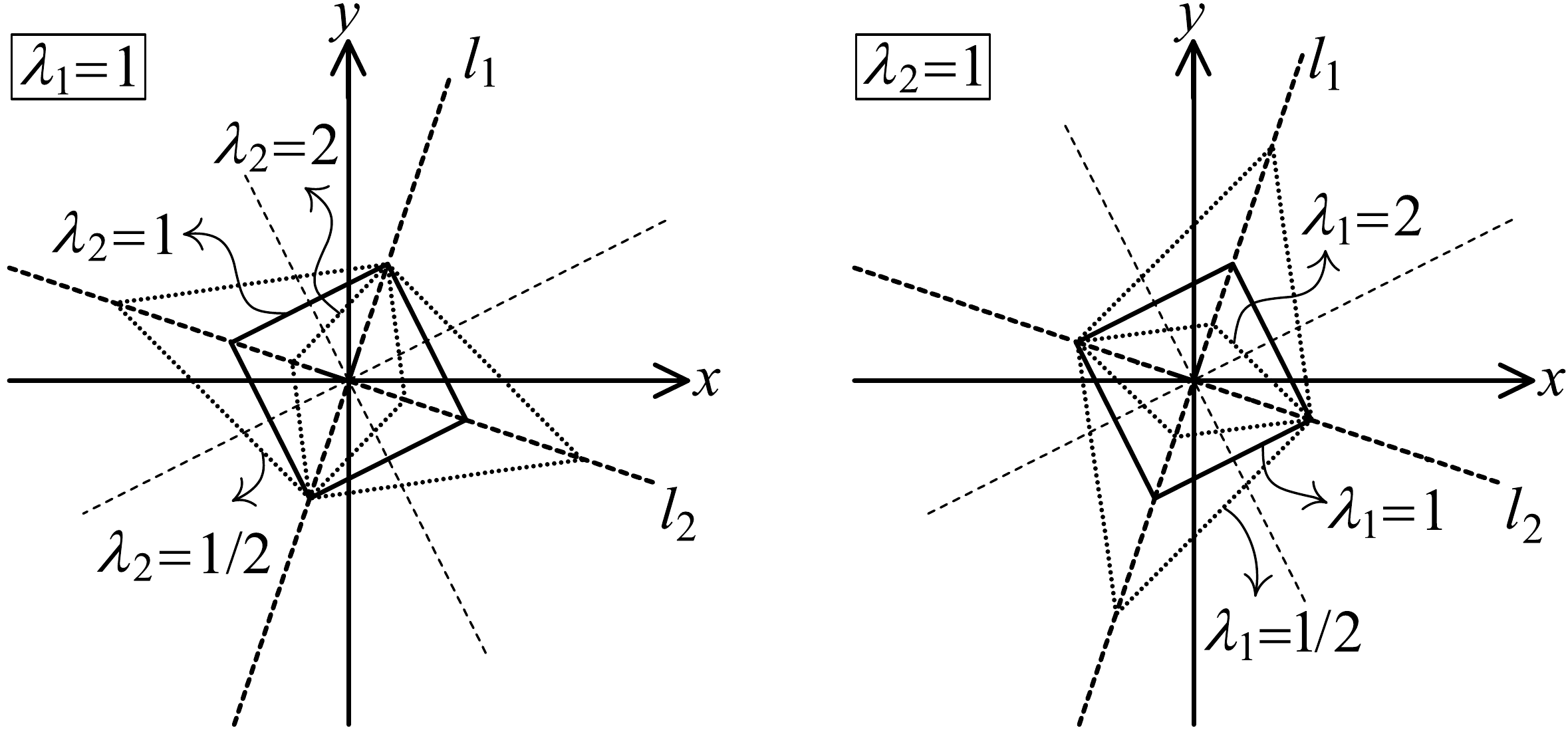} \vspace{-0.05in}
\end{center}
\begin{center}
\textbf{Figure 3.} The unit $(v_{1},v_{2})$-taxicab circles\ for 
$v_{1}=\left( \frac{3}{\sqrt{10}},\frac{-1}{\sqrt{10}}\right) $, 
$v_{2}=\left( \frac{1}{\sqrt{10}},\frac{3}{\sqrt{10}}\right) $. \smallskip 
\end{center}

\noindent Let us consider the case of $\lambda _{1}=\lambda _{2}=1$: Now, we
know that a $(v_{1},v_{2})$-taxicab circle with center $C$ and radius $r$,
that is the set of all points $P$ satisfying the equation 
\begin{equation*}
d_{E}(P,l_{1})+d_{E}(P,l_{2})=r,
\end{equation*}%
is a rectangle with the same center, whose diagonals are on the lines $l_{1}$
and $l_{2}$, and sides are parallel to angle bisectors of the lines $l_{1}$
and $l_{2}$. Besides, if $v_{1}\perp v_{2}$ then $(v_{1},v_{2})$-taxicab
circle is a square with the same properties. On the other hand, for a point 
$Q_{i}$ on both line $l_{i}$ and the $(v_{1},v_{2})$-taxicab circle (see
Figure 4), it is clear that 
\begin{equation*}
d_{E}(Q_{1},l_{2})=d_{E}(Q_{2},l_{1})=r.
\end{equation*}

\begin{center}
\includegraphics[width=3.5 in]{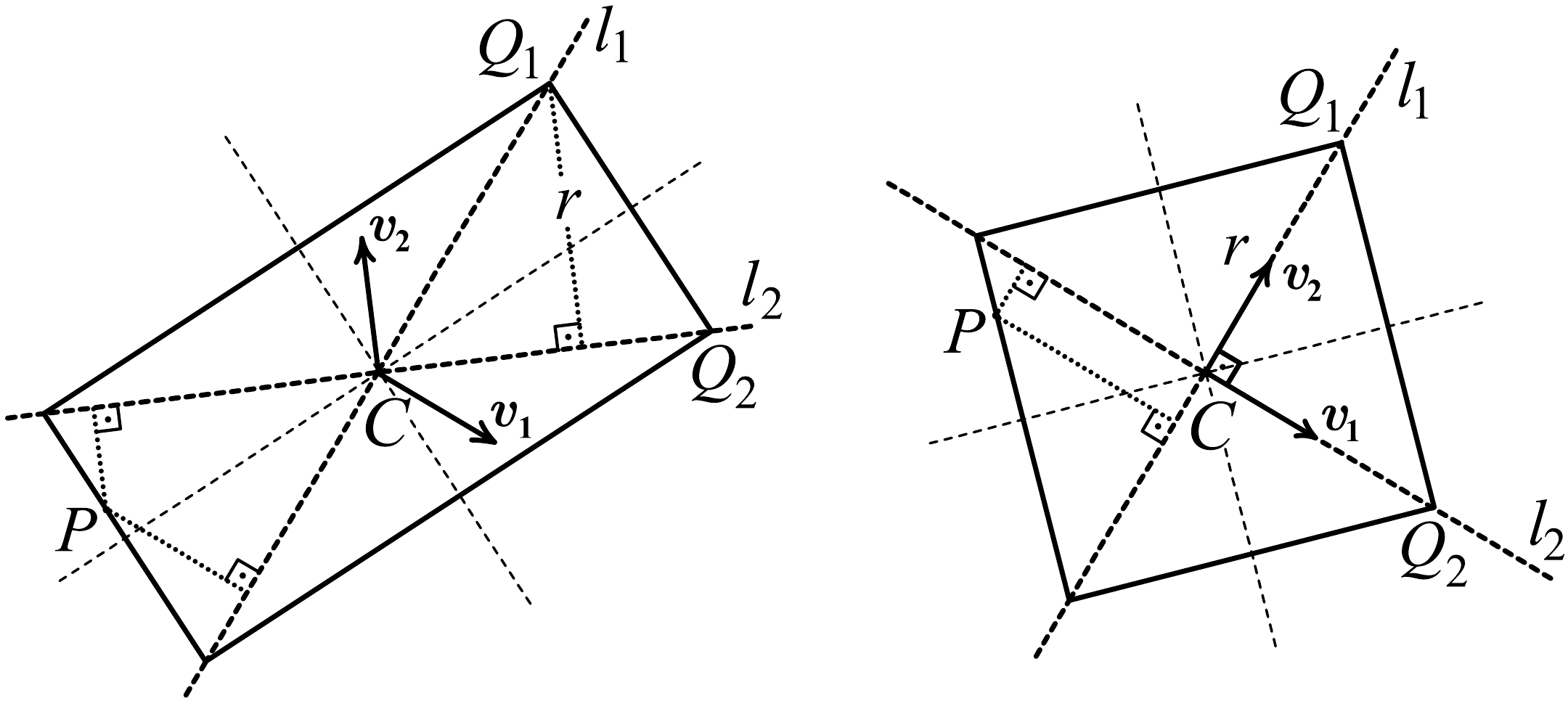} \vspace{-0.1in}
\end{center}
\begin{center}
\textbf{Figure 4.} $(v_{1},v_{2})$-taxicab circles with center $C$ and radius $r$, 
for $\lambda _{1}=\lambda _{2}=1$.
\end{center}

\noindent The following theorem shows that every rectangle is a 
$(v_{1},v_{2})$-taxicab circle with the same center for a proper generalized
taxicab metric $d_{T(v_{1},v_{2})}$ with $\lambda _{1}=\lambda _{2}=1$:

\begin{theorem}
Every rectangle with sides of lengths $2a$ and $2b$, is a 
$(v_{1},v_{2})$-taxicab circle with the same center and the radius 
$\frac{2ab}{\sqrt{a^{2}+b^{2}}}$, for $\lambda _{1}=\lambda _{2}=1$ and 
linearly independent unit vectors $v_{1}$ and $v_{2}$, each of which is 
perpendicular to a diagonal of the rectangle.
\end{theorem}

\begin{proof}
Without loss of generality, let us consider a rectangle with center $C$ and
sides of lengths $2a$ and $2b$, as in Figure 5. Denote the diagonal lines of
the rectangle by $d_{1}$ and $d_{2}$. Clearly, $C$ is the intersection point
of $d_{1}$ and $d_{2}$. Draw two lines $d_{2}^{\prime }$ and $d_{2}^{\prime
\prime }$, each of them is passing through a vertex on $d_{1}$ and parallel
to $d_{2}$. Since sides of the rectangle are angle bisectors of pair of
lines $d_{1},d_{2}^{\prime }$ and $d_{1},d_{2}^{\prime \prime }$, we have 
\begin{equation}
d_{E}(P,d_{1})+d_{E}(P,d_{2})=d_{E}(d_{2},d_{2}^{\prime
})=d_{E}(d_{2},d_{2}^{\prime \prime }).
\end{equation}
On the other hand, for the area of the rectangle we have
\begin{equation}
4ab=2\sqrt{a^{2}+b^{2}}d_{E}(d_{2},d_{2}^{\prime }),
\end{equation}
so we get 
\begin{equation}
d_{E}(d_{2},d_{2}^{\prime })=\frac{2ab}{\sqrt{a^{2}+b^{2}}}.
\end{equation}
Then, for every point $P$ on the rectangle, we have
\begin{equation}
d_{E}(P,d_{1})+d_{E}(P,d_{2})=\frac{2ab}{\sqrt{a^{2}+b^{2}}}.
\end{equation}
Thus, for $\lambda _{1}=\lambda _{2}=1$ and linearly independent unit
vectors $v_{1}$ and $v_{2}$, each of which is perpendicular to a diagonal,
the rectangle is a $(v_{1},v_{2})$-taxicab circle with center $C$ and radius 
$2ab/\sqrt{a^{2}+b^{2}}$.\vspace{-0.1in}
\end{proof}

\begin{center}
\includegraphics[width=2 in]{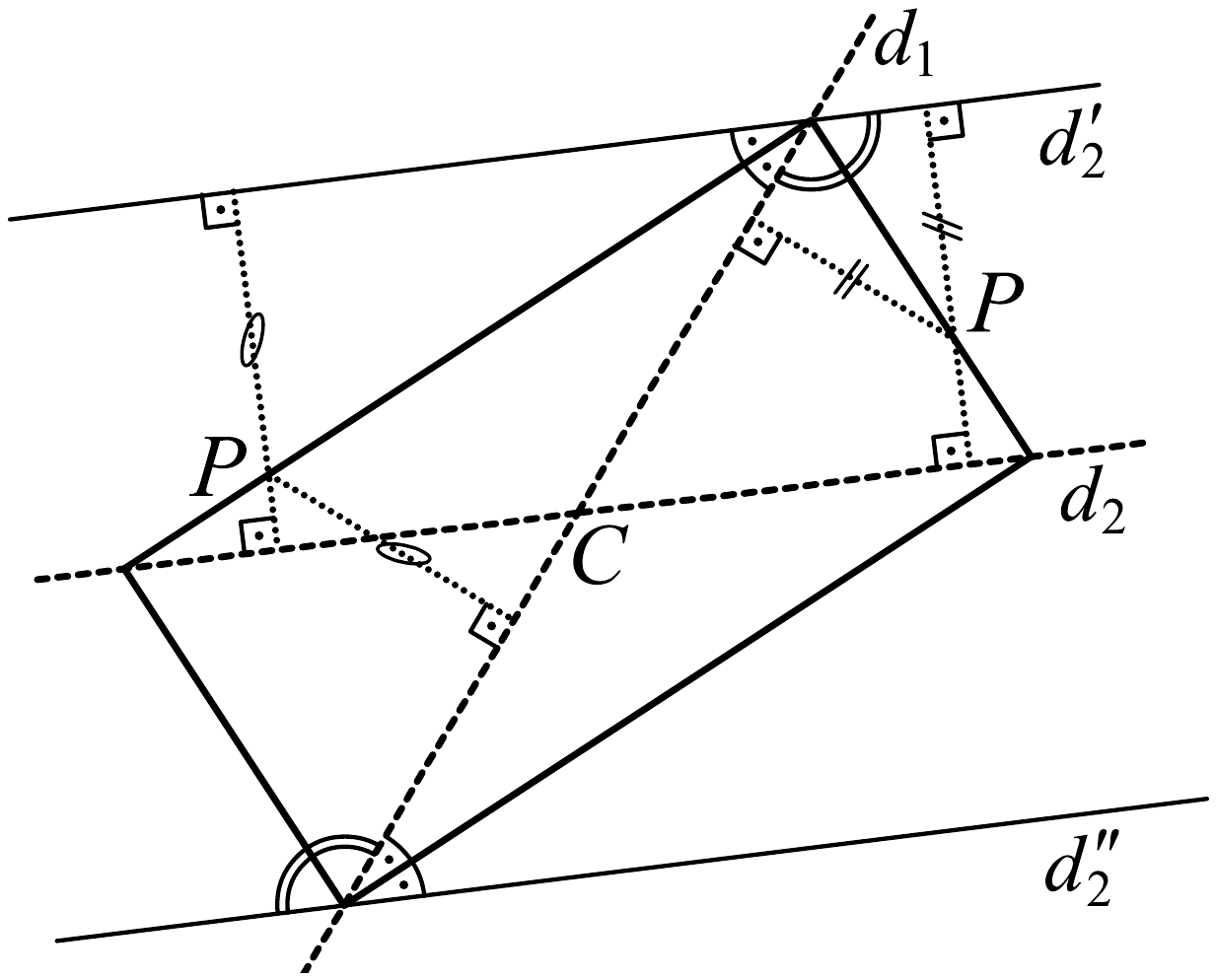}
\end{center}
\begin{center}
\textbf{Figure 5.} A rectangle with center $C$ and
sides of lengths $2a$ and $2b$.
\end{center}

\section{Circles of the generalized maximum metric in $\mathbb{R}^{2}$}

\noindent By Definition 2.1 and Remark 2.1, $(v_{1},v_{2})$-maximum distance
between points $P_{1}=(x_{1},y_{1})$ and $P_{2}=(x_{2},y_{2})$ in $\mathbb{R}^{2}$ is
\begin{eqnarray*}
d_{M(v_{1},v_{2})}(P_{1},P_{2}) &=&\max \{\lambda _{1}\left\vert
v_{11}(x_{1}-x_{2})+v_{12}(y_{1}-y_{2})\right\vert ,\lambda _{2}\left\vert
v_{21}(x_{1}-x_{2})+v_{22}(y_{1}-y_{2})\right\vert \} \\
&=&\max \{\lambda _{1}\text{\thinspace }d_{E}(P_{2},l_{1}),\lambda _{2}\text{%
\thinspace }d_{E}(P_{2},l_{2})\}
\end{eqnarray*}
that is the maximum of weighted Euclidean distances from the point $P_{2}$
to the lines $l_{1}$ and $l_{2}$, which are passing through$\ P_{1}$ and
perpendicular to the vectors $v_{1}$ and $v_{2}$ respectively. \medskip

\noindent The following theorem determines circles of the generalized
maximum metric $d_{M(v_{1},v_{2})}$ in $\mathbb{R}^{2}$:

\begin{theorem}
Every $(v_{1},v_{2})$-maximum circle is a parallelogram with the same
center, each of whose sides is perpendicular to $v_{1}$ or $v_{2}$. In
addition, if $\lambda _{1}=\lambda _{2}$ then it is a rhombus, if $%
v_{1}\perp v_{2}$ then it is a rectangle, and if $\lambda _{1}=\lambda _{2}$
and $v_{1}\perp v_{2}$ then it is a square.
\end{theorem}

\begin{proof}
Without loss of generality, let us consider\ the unit 
$(v_{1},v_{2})$-maximum circle. Clearly, it is the set of points 
$P=(x,y)$ in $\mathbb{R}^{2}$ satisfying the equation 
\begin{equation}
d_{M(v_{1},v_{2})}(O,P)=\max \left\{ \lambda _{1}\,d_{E}(P,l_{1}),\lambda
_{2}\,d_{E}(P,l_{2})\right\} =1
\end{equation}%
where $l_{i}:v_{i1}x+v_{i2}y=0$ for $i=1,2$, that is
\begin{equation}
\max \left\{ \lambda _{1}\left\vert v_{11}x_{1}+v_{12}x_{2}\right\vert
,\lambda _{2}\left\vert v_{21}x_{1}+v_{22}x_{2}\right\vert \right\} =1.
\end{equation}
One can see that this equation is the image of $\max \left\{ \left\vert
x\right\vert ,\left\vert y\right\vert \right\} =1$ which is the well-known
maximum circle, under the linear transformation
\begin{equation}
T\left( \left[ 
\begin{array}{c}
x \\ 
y%
\end{array}%
\right] \right) =\left[ 
\begin{array}{cc}
\frac{v_{22}}{\lambda _{1}\tau } & \frac{-v_{12}}{\lambda _{2}\tau } \\ 
\frac{-v_{21}}{\lambda _{1}\tau } & \frac{v_{11}}{\lambda _{2}\tau }%
\end{array}%
\right] \left[ 
\begin{array}{c}
x \\ 
y%
\end{array}%
\right]
\end{equation}
where $\tau =\left\vert 
\begin{array}{cc}
v_{11} & v_{12} \\ 
v_{21} & v_{22}%
\end{array}%
\right\vert $. Thus, the unit $(v_{1},v_{2})$-maximum circle is a
parallelogram symmetric about the origin, having vertices $B_{1}=\left( 
\frac{-v_{12}\lambda _{1}+v_{22}\lambda _{2}}{\lambda _{1}\lambda _{2}\tau },%
\frac{v_{11}\lambda _{1}-v_{21}\lambda _{2}}{\lambda _{1}\lambda _{2}\tau }%
\right) $, $B_{2}=\left( \frac{-v_{12}\lambda _{1}-v_{22}\lambda _{2}}{%
\lambda _{1}\lambda _{2}\tau },\frac{v_{11}\lambda _{1}+v_{21}\lambda _{2}}{%
\lambda _{1}\lambda _{2}\tau }\right) $, \newline
$B_{3}=\left( \frac{v_{12}\lambda _{1}-v_{22}\lambda _{2}}{\lambda
_{1}\lambda _{2}\tau },\frac{-v_{11}\lambda _{1}+v_{21}\lambda _{2}}{\lambda
_{1}\lambda _{2}\tau }\right) $, $B_{4}=\left( \frac{v_{12}\lambda
_{1}+v_{22}\lambda _{2}}{\lambda _{1}\lambda _{2}\tau },\frac{-v_{11}\lambda
_{1}-v_{21}\lambda _{2}}{\lambda _{1}\lambda _{2}\tau }\right) $, and having
sides parallel to the lines $l_{1}$ and $l_{2}$, each of which is
perpendicular to $v_{1}$ or $v_{2}$, since 
\begin{equation*}
B_{1}B_{2}//B_{3}B_{4}//l_{2}\text{\ and }B_{1}B_{4}//B_{2}B_{3}//l_{1}.
\end{equation*}
In addition, if $\lambda _{1}=\lambda _{2}$ then $OB_{1}\perp OB_{2}$ and
since a parallelogram having perpendicular diagonals is a rhombus, the unit 
$(v_{1},v_{2})$-maximum circle is a rhombus. Notice that diagonals of the
rhombus are on angle bisectors of the lines $l_{1}$ and $l_{2}$. If 
$v_{1}\perp v_{2}$ then $l_{1}\perp l_{2}$ and since a parallelogram having
perpendicular sides is a rectangle, the unit $(v_{1},v_{2})$-maximum circle
is a rectangle. Finally, it is clear that if $\lambda _{1}=\lambda _{2}$ and 
$v_{1}\perp v_{2}$, then the unit $(v_{1},v_{2})$-maximum circle is a square
(see Figure 6 and Figure 7 for examples of the unit $(v_{1},v_{2})$-maximum
circles). 
\end{proof}

\begin{center}
\includegraphics[width=4.2 in]{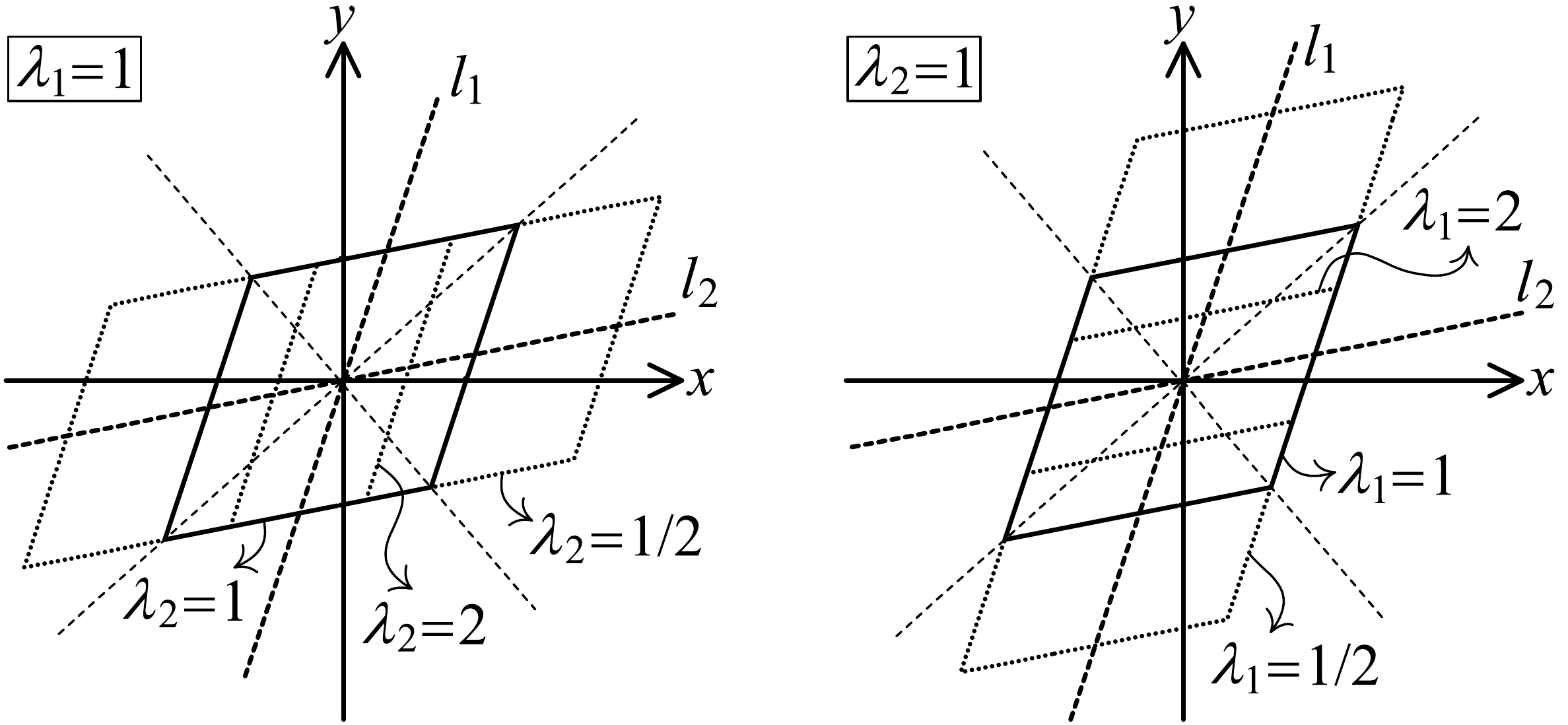}
\end{center}
\begin{center}
\textbf{Figure 6.} The unit $(v_{1},v_{2})$-maximum circles; 
$v_{1}=\left( \frac{3}{\sqrt{10}},\frac{-1}{\sqrt{10}}\right) $, 
$v_{2}=\left( \frac{-1}{\sqrt{26}},\frac{5}{\sqrt{26}}\right) $.\smallskip
\end{center}

\begin{center}
\includegraphics[width=4.2 in]{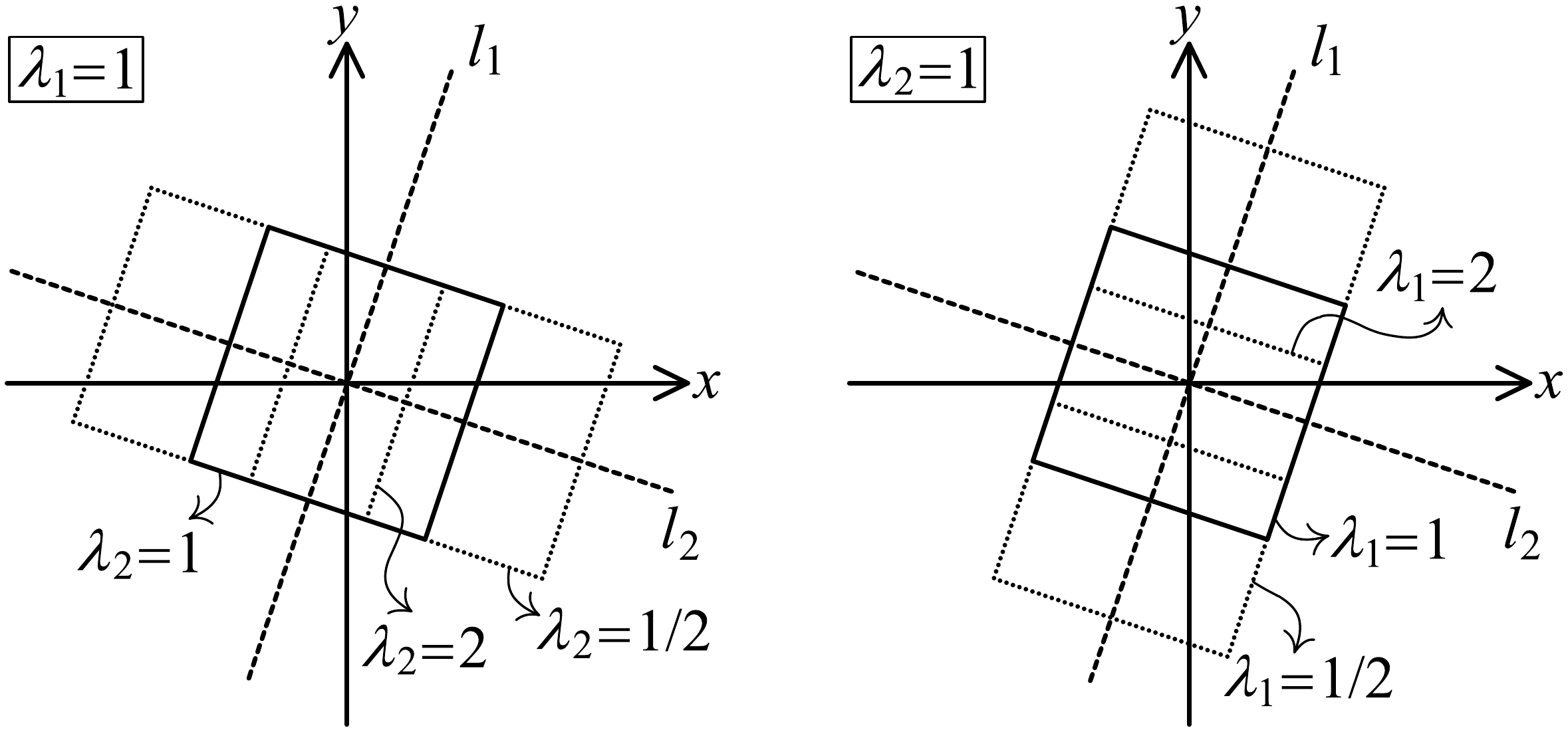} \vspace{-0.18in}
\end{center}
\begin{center}
\textbf{Figure 7.} The unit $(v_{1},v_{2})$-maximum circles\ for 
$v_{1}=\left( \frac{3}{\sqrt{10}},\frac{-1}{\sqrt{10}}\right) $, 
$v_{2}=\left( \frac{1}{\sqrt{10}},\frac{3}{\sqrt{10}}\right) $.\smallskip
\end{center}

\noindent Now let us consider the case of $\lambda _{1}=\lambda _{2}=1$:
Now, we know that a $(v_{1},v_{2})$-maximum circle with center $C$ and
radius $r$, that is the set of all points $P$ satisfying the equation  \vspace{-0.08in}
\begin{equation*}
\max \{d_{E}(P,l_{1}),d_{E}(P,l_{2})\}=r, \vspace{-0.08in}
\end{equation*}
is a rhombus with the same center, whose sides are parallel to the lines 
$l_{1}$ and $l_{2}$, and whose diagonals are on angle bisectors of the lines 
$l_{1}$ and $l_{2}$. Besides, if $v_{1}\perp v_{2}$ then the 
$(v_{1},v_{2})$-maximum circle is a square with the same properties. 
On the other hand, for a point $Q_{i}$ on both line $l_{i}$ and the 
$(v_{1},v_{2})$-maximum circle (see Figure 8), it is clear that  \vspace{-0.08in}
\begin{equation*}
d_{E}(Q_{1},l_{2})=d_{E}(Q_{2},l_{1})=r. \vspace{-0.08in}
\end{equation*}

\begin{center}
\includegraphics[width=3.3 in]{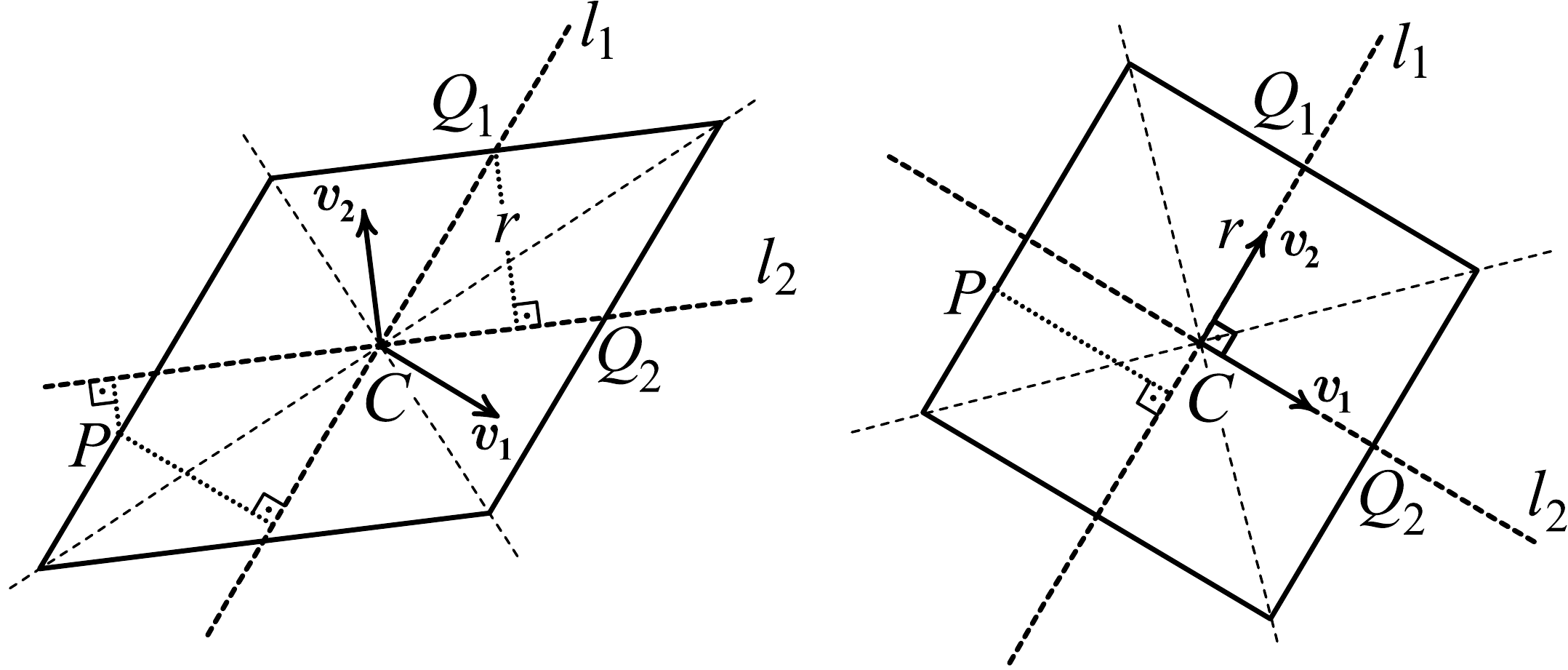} \vspace{-0.15in}
\end{center}
\begin{center}
\textbf{Figure 8.} $(v_{1},v_{2})$-maximum circles with center $C$ and radius $r$, for $\lambda _{1}=\lambda _{2}=1$. \medskip
\end{center}

\noindent The following theorem shows that every rhombus is a 
$(v_{1},v_{2})$-maximum circle with the same center for a proper generalized maximum metric 
$d_{M(v_{1},v_{2})}$ with $\lambda _{1}=\lambda _{2}=1$:

\begin{theorem}
Every rhombus with diagonals of lengths $2e$ and $2f$, is a 
$(v_{1},v_{2})$-maximum circle with the same center and the radius 
$\frac{ef}{\sqrt{e^{2}+f^{2}}}$, for $\lambda _{1}=\lambda _{2}=1$ and linearly independent
unit vectors $v_{1}$ and $v_{2}$, each of which is perpendicular to a side
of the rhombus.
\end{theorem}

\begin{proof}
Without loss of generality, let us consider a rhombus with center $C$ and
diagonals of lengths $2e$ and $2f$, as in Figure 9. Denote by $d_{1}$ and 
$d_{2}$, the two distinct lines each through $C$ and parallel to a side of
the rhombus. Since diagonals are angle bisectors of consecutive sides, we
have 
\begin{equation}
\max \{d_{E}(P,d_{1}),d_{E}(P,d_{2})\}=d_{E}(V,d_{1})=d_{E}(V,d_{2})
\end{equation}%
for any vertex $V$ of the rhombus. On the other hand, for the area of the
rhombus we have 
\begin{equation}
2ef=2\sqrt{e^{2}+f^{2}}d_{E}(V,d_{1}),
\end{equation}
so, we get 
\begin{equation}
d_{E}(V,d_{1})=\frac{ef}{\sqrt{e^{2}+f^{2}}}.
\end{equation}
Then, for every point $P$ on the rhombus we have 
\begin{equation}
\max \{d_{E}(P,d_{1}),d_{E}(P,d_{2})\}=\frac{ef}{\sqrt{e^{2}+f^{2}}}.
\end{equation}%
Thus, for $\lambda _{1}=\lambda _{2}=1$ and linearly independent unit
vectors $v_{1}$ and $v_{2}$, each of which is perpendicular to a side, the
rhombus is a $(v_{1},v_{2})$-maximum circle having center $C$ and radius $ef/\sqrt{e^{2}+f^{2}}$.
\end{proof}

\begin{center}
\includegraphics[width=1.7 in]{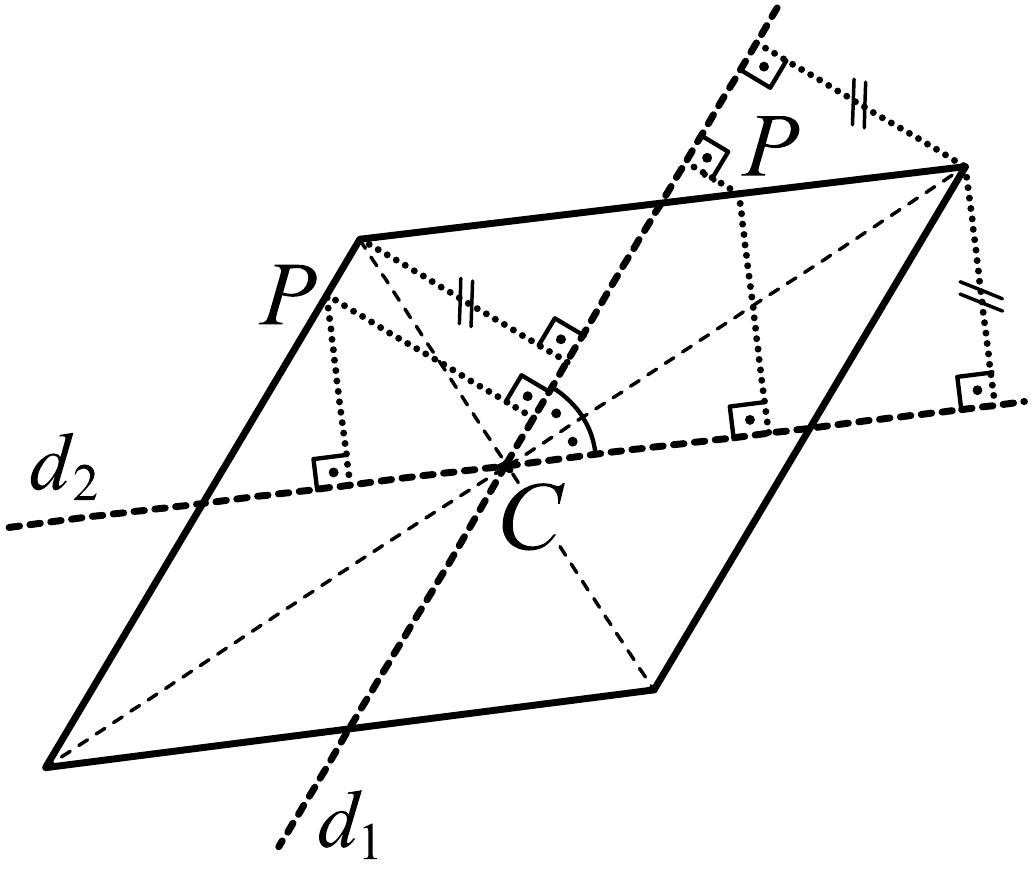} \vspace{-0.1in}
\end{center}
\begin{center}
\textbf{Figure 9.} A rhombus with center $C$ and
diagonals of lengths $2e$ and $2f$.
\end{center}

\section{Circles of the generalized Euclidean metric in $\mathbb{R}^{2}$}

\noindent By Definition 2.1 and Remark 2.1, $(v_{1},v_{2})$-Euclidean
distance between points $P_{1}=(x_{1},y_{1})$ and $P_{2}=(x_{2},y_{2})$ in $%
\mathbb{R}^{2}$ is%
\begin{eqnarray*}
d_{E(v_{1},v_{2})}(P_{1},P_{2}) &=&\left[ (\lambda _{1}\left\vert
v_{11}(x_{1}-x_{2})+v_{12}(y_{1}-y_{2})\right\vert )^{2}+(\lambda
_{2}\left\vert v_{21}(x_{1}-x_{2})+v_{22}(y_{1}-y_{2})\right\vert )^{2}%
\right] ^{1/2} \\
&=&\left[ (\lambda _{1}\,d_{E}(P_{2},l_{1}))^{2}+(\lambda
_{2}\,d_{E}(P_{2},l_{2}))^{2}\right] ^{1/2}
\end{eqnarray*}%
that is the square root of the sum of square of weighted Euclidean distances
from the points $P_{2}$ to the lines $l_{1}$ and $l_{2}$, which are passing
through$\ P_{1}$ and perpendicular to the vectors $v_{1}$ and $v_{2}$
respectively. Notice that by Pythagorean theorem, for $\lambda _{1}=\lambda
_{2}=1$ and perpendicular unit vectors $v_{1}$ and $v_{2}$, we have 
\begin{equation}
d_{E(v_{1},v_{2})}(P_{1},P_{2})=d_{E}(P_{1},P_{2}).
\end{equation}%
\noindent The following theorem determines circles of the
generalized Euclidean metric $d_{E(v_{1},v_{2})}$ in $\mathbb{R}^{2}$:

\begin{theorem}
Every $(v_{1},v_{2})$-Euclidean circle is an ellipse with the same center.
In addition, if $\lambda _{1}=\lambda _{2}$ then its axes are angle
bisectors of the lines $l_{1}$ and $l_{2}$, if $v_{1}\perp v_{2}$ then its
axes are the lines $l_{1}$ and $l_{2}$, and if $\lambda _{1}=\lambda _{2}$
and $v_{1}\perp v_{2}$ then it is a Euclidean circle with the same center.
\end{theorem}

\begin{proof}
Without loss of generality, we consider the unit $(v_{1},v_{2})$-Euclidean
circle with center at the origin. Clearly, it is the set of points $P=(x,y)$
in $\mathbb{R}^{2}$ satisfying the equation 
\begin{equation}
d_{E(v_{1},v_{2})}(O,P)=\left[ (\lambda _{1}\,d_{E}(P,l_{1}))^{2}+(\lambda
_{2}\,d_{E}(P,l_{2}))^{2}\right] ^{1/2}=1
\end{equation}%
where $l_{i}:v_{i1}x+v_{i2}y=0$ for $i=1,2$, that is
\begin{equation}
\lambda _{1}^{2}\left( v_{11}x+v_{12}y\right) ^{2}+\lambda _{2}^{2}\left(
v_{21}x+v_{22}y\right) ^{2}=1.
\end{equation}%
\noindent This equation can be written as
\begin{equation}
Ax^{2}+By^{2}+2Cxy+2Dx+2Ey+F=0
\end{equation}
where $A=\lambda _{1}^{2}v_{11}^{2}+\lambda _{2}^{2}v_{21}^{2}$, $B=\lambda
_{1}^{2}v_{12}^{2}+\lambda _{2}^{2}v_{22}^{2}$, $C=\lambda
_{1}^{2}v_{11}v_{12}+\lambda _{2}^{2}v_{21}v_{22}$, $D=E=0$ and $F=-1$.\smallskip \newline
If we use the classification conditions for the general quadratic equations
in two variables (see \cite[pp. 232-233]{Zwil}), we have%
\begin{equation*}
\delta =\left\vert 
\begin{array}{cc}
A & C \\ 
C & B%
\end{array}%
\right\vert =\lambda _{1}^{2}\lambda _{2}^{2}\tau ^{2}\text{ \ and \ }\Delta
=\left\vert 
\begin{array}{ccc}
A & C & 0 \\ 
C & B & 0 \\ 
0 & 0 & -1%
\end{array}%
\right\vert =-\delta
\end{equation*}%
where $\tau =\left\vert 
\begin{array}{cc}
v_{11} & v_{21} \\ 
v_{12} & v_{22}%
\end{array}%
\right\vert $, and since\ $v_{1}$ and $v_{2}$ are linearly independent, we
get $\tau \neq 0$, $\delta >0$ and $\Delta <0$. In addition, since $A>0$ and 
$B>0$, we get $\Delta /(A+B)<0$. So, since 
\begin{equation*}
\Delta \neq 0,\delta >0\text{ \ and \ }\Delta /(A+B)<0,
\end{equation*}%
the quadratic equation determines an ellipse with center at the origin. If 
$\lambda _{1}=\lambda _{2}$, concerning the equation (28) geometrically, one
can see that the unit $(v_{1},v_{2})$-Euclidean circle is symmetric about
angle bisectors of the lines $l_{1}$ and $l_{2}$, since $l_{1}$ and $l_{2}$
are symmetric about the angle bisectors of themselves. Notice that the major
axis of the ellipse is the angle bisector of the non-obtuse angle between 
$l_{1}$ and $l_{2}$. If $v_{1}\perp v_{2}$ then $l_{1}\perp l_{2}$, and
concerning the equation (28) geometrically again, one can see that the unit 
$(v_{1},v_{2})$-Euclidean circle is symmetric about the lines $l_{1}$ and 
$l_{2}$. Finally, it is clear that if $\lambda _{1}=\lambda _{2}$ and 
$v_{1}\perp v_{2}$, then the unit $(v_{1},v_{2})$-Euclidean circle is
Euclidean circle with the same center, since an ellipse which is symmetric
about four different lines ($l_{1}$, $l_{2}$ and angle bisectors of them),
is a Euclidean circle. One can also see that if $\lambda _{1}=\lambda _{2}$
and $v_{1}\perp v_{2}$, then $A=B>0$ and $C=0$, so the quadratic equation
above gives an equation of a Euclidean circle with center at the origin (see
Figure 10 and Figure 11 for examples of the unit $(v_{1},v_{2})$-Euclidean
circles).\medskip
\end{proof}

\begin{center}
\includegraphics[width=4.2 in]{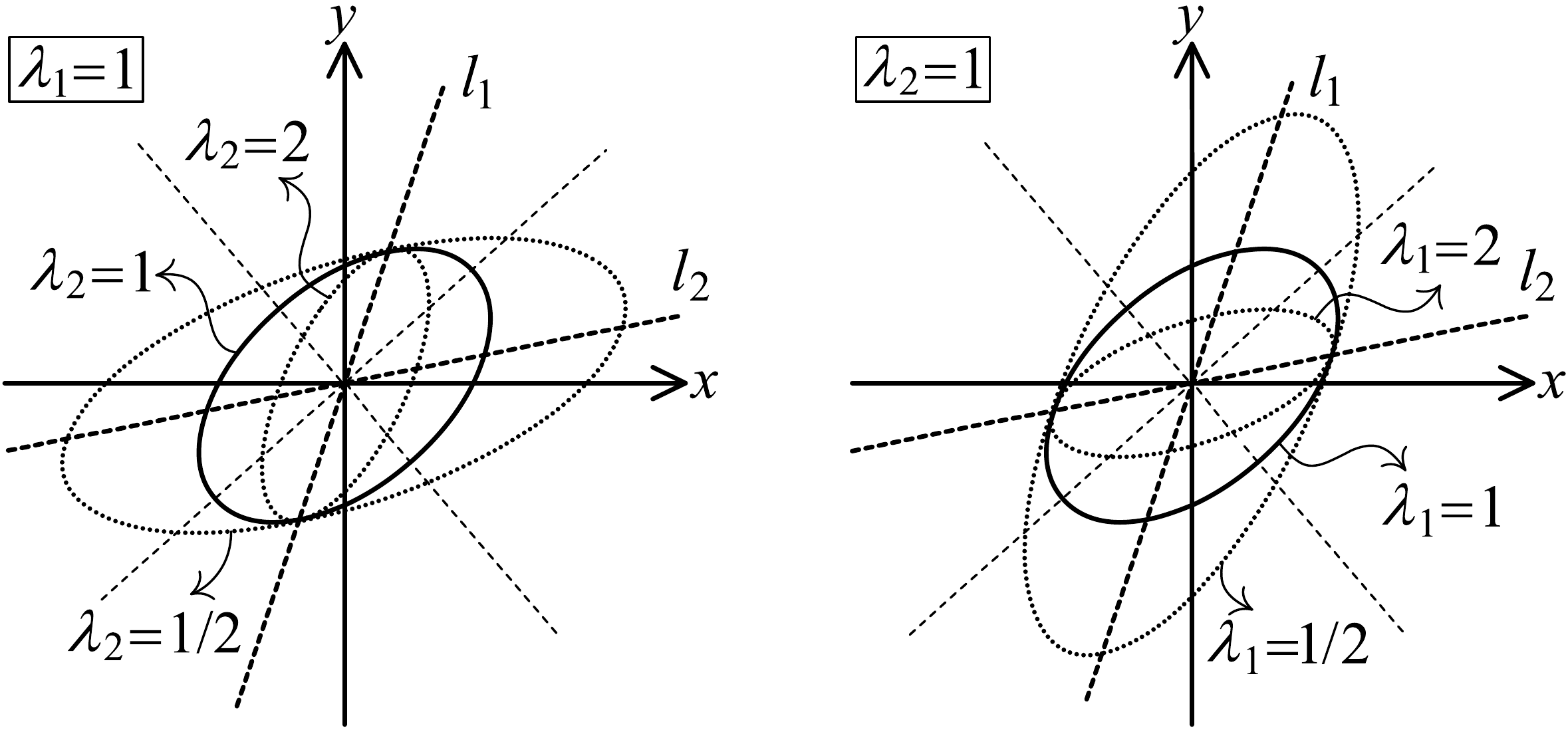} \vspace{-0.1in}
\end{center}
\begin{center}
\textbf{Figure 10.} The unit $(v_{1},v_{2})$-Euclidean circles for 
$v_{1}=\left( \frac{3}{\sqrt{10}},\frac{-1}{\sqrt{10}}\right) $, 
$v_{2}=\left( \frac{-1}{\sqrt{26}},\frac{5}{\sqrt{26}}\right) $.\smallskip
\end{center}

\begin{center}
\includegraphics[width=4.2 in]{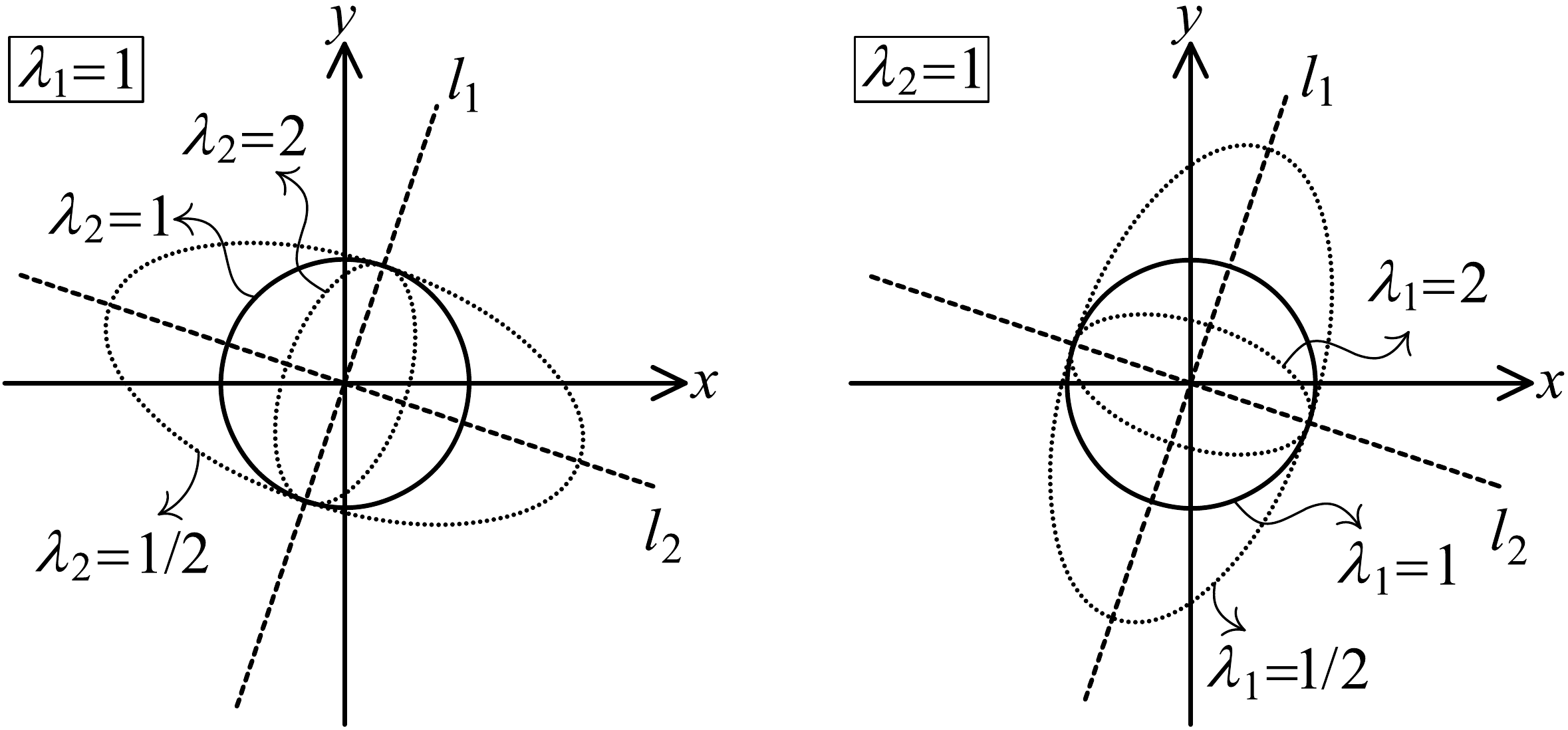} \vspace{-0.1in}
\end{center}
\begin{center}
\textbf{Figure 11.} The unit $(v_{1},v_{2})$-Euclidean circles\ for 
$v_{1}=\left( \frac{3}{\sqrt{10}},\frac{-1}{\sqrt{10}}\right) $, 
$v_{2}=\left( \frac{1}{\sqrt{10}},\frac{3}{\sqrt{10}}\right) $.\smallskip
\end{center}

\noindent Let us consider the case of $\lambda _{1}=\lambda _{2}=1$: Now, we
know that a $(v_{1},v_{2})$-Euclidean circle with center $C$ and radius $r$,
that is the set of all points $P$ satisfying the equation 
\begin{equation*}
\left[ (d_{E}(P,l_{1}))^{2}+(d_{E}(P,l_{2}))^{2}\right] ^{1/2}=r,
\end{equation*}%
is an ellipse with the same center, whose axes are angle bisectors of the
lines $l_{1}$ and $l_{2}$, such that the major axis is the angle bisector of
the non-obtuse angle between $l_{1}$ and $l_{2}$. In addition, if 
$v_{1}\perp v_{2}$ then $(v_{1},v_{2})$-Euclidean circle is a Euclidean
circle having the same center and the radius. In addition, for a point 
$Q_{i} $ on both line $l_{i}$ and the $(v_{1},v_{2})$-Euclidean circle (see
Figure 12), it is clear that 
\begin{equation*}
d_{E}(Q_{1},l_{2})=d_{E}(Q_{2},l_{1})=r.
\end{equation*}

\begin{center}
\includegraphics[width=3.5 in]{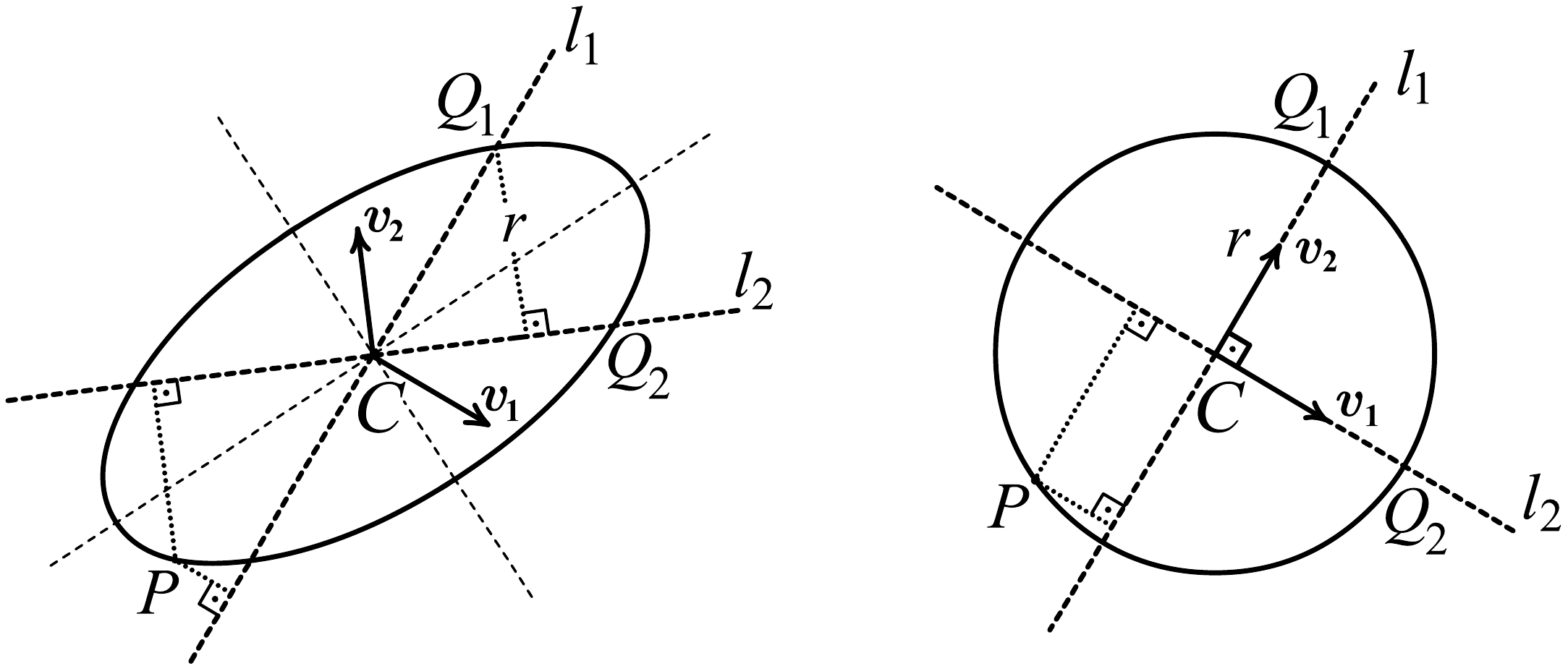} \vspace{-0.1in}
\end{center}
\begin{center}
\textbf{Figure 12.} $(v_{1},v_{2})$-Euclidean circles with center $C$ and radius $r$, for $\lambda _{1}=\lambda _{2}=1$. \smallskip 
\end{center}

\noindent The following theorem determines some relations between parameters
of a $(v_{1},v_{2})$-Euclidean circle and the ellipse related to it:

\begin{theorem}
If a $(v_{1},v_{2})$-Euclidean circle with radius $r$ for $\lambda
_{1}=\lambda _{2}=1$, is an ellipse with the same center, having semi-major
axis $a$ and semi-minor axis $b$, then\vspace{-0.07in} 
\begin{equation*}
r=\frac{\sqrt{2}ab}{\sqrt{a^{2}+b^{2}}}\text{, \ }a=\frac{r}{\sqrt{1-\cos
\theta }}\text{ \ and \ }b=\frac{r}{\sqrt{1+\cos \theta }}
\end{equation*}%
where $\theta $ is the non-obtuse angle between\ $v_{1}$ and $v_{2}$, and 
$\cos \theta =\left\vert v_{11}v_{21}+v_{12}v_{22}\right\vert .$
\end{theorem}

\begin{proof}
Let a $(v_{1},v_{2})$-Euclidean circle with radius $r$ for $\lambda
_{1}=\lambda _{2}=1$, be an ellipse with the same center, having semi-major
axis $a$ and semi-minor axis $b$, and let $\theta $ be the non-obtuse angle
between\ $v_{1}$ and $v_{2}$. Then non-obtuse angle between the lines $l_{1}$
and $l_{2}$ is equal to $\theta $, and the axes of the ellipse is angle
bisectors of the lines $l_{1}$ and $l_{2}$, such that the major axis of is
the angle bisector of the non-obtuse angle between $l_{1}$ and $l_{2}$.
Using similar right triangles whose hypotenuses are $a$ and $b$ (see Figure
13), one gets $\sin \frac{\theta }{2}=\frac{r}{a\sqrt{2}}$, 
$\cos \frac{\theta }{2}=\frac{r}{b\sqrt{2}}$, 
$\tan \frac{\theta }{2}=\frac{r}{\sqrt{2a^{2}-r^{2}}}=\frac{\sqrt{2b^{2}-r^{2}}}{r}$, 
and so
\begin{equation*}
\tan \frac{\theta }{2}=\frac{b}{a}\text{, }\sin \theta =\frac{r^{2}}{ab}%
\text{ and \ }\cos \theta =1-\frac{r^{2}}{a^{2}}=\frac{r^{2}}{b^{2}}-1.
\end{equation*}
Then, we have
\begin{equation*}
r=\frac{\sqrt{2}ab}{\sqrt{a^{2}+b^{2}}}\text{, }a=\frac{r}{\sqrt{1-\cos
\theta }}\text{ \ and \ }b=\frac{r}{\sqrt{1+\cos \theta }}.
\end{equation*}
Besides, one can derive that
\begin{equation*}
\sin \theta =\frac{2ab}{\sqrt{a^{2}+b^{2}}}\text{,\ }\cos \theta =\frac{%
a^{2}-b^{2}}{a^{2}+b^{2}}\text{ and }\tan \theta =\frac{2ab}{a^{2}-b^{2}}.
\end{equation*}
In addition, by $\left\vert \left\langle v_{1},v_{2}\right\rangle
\right\vert =\left\Vert v_{1}\right\Vert \left\Vert v_{2}\right\Vert \cos
\theta $ it follows immediately that 
\begin{equation*}
\cos \theta =\left\vert v_{11}v_{21}+v_{12}v_{22}\right\vert \text{ \ and \ }%
\sin \theta =\left\vert v_{11}v_{22}-v_{12}v_{21}\right\vert .
\end{equation*}
\end{proof}

\begin{center}
\includegraphics[width=2.1 in]{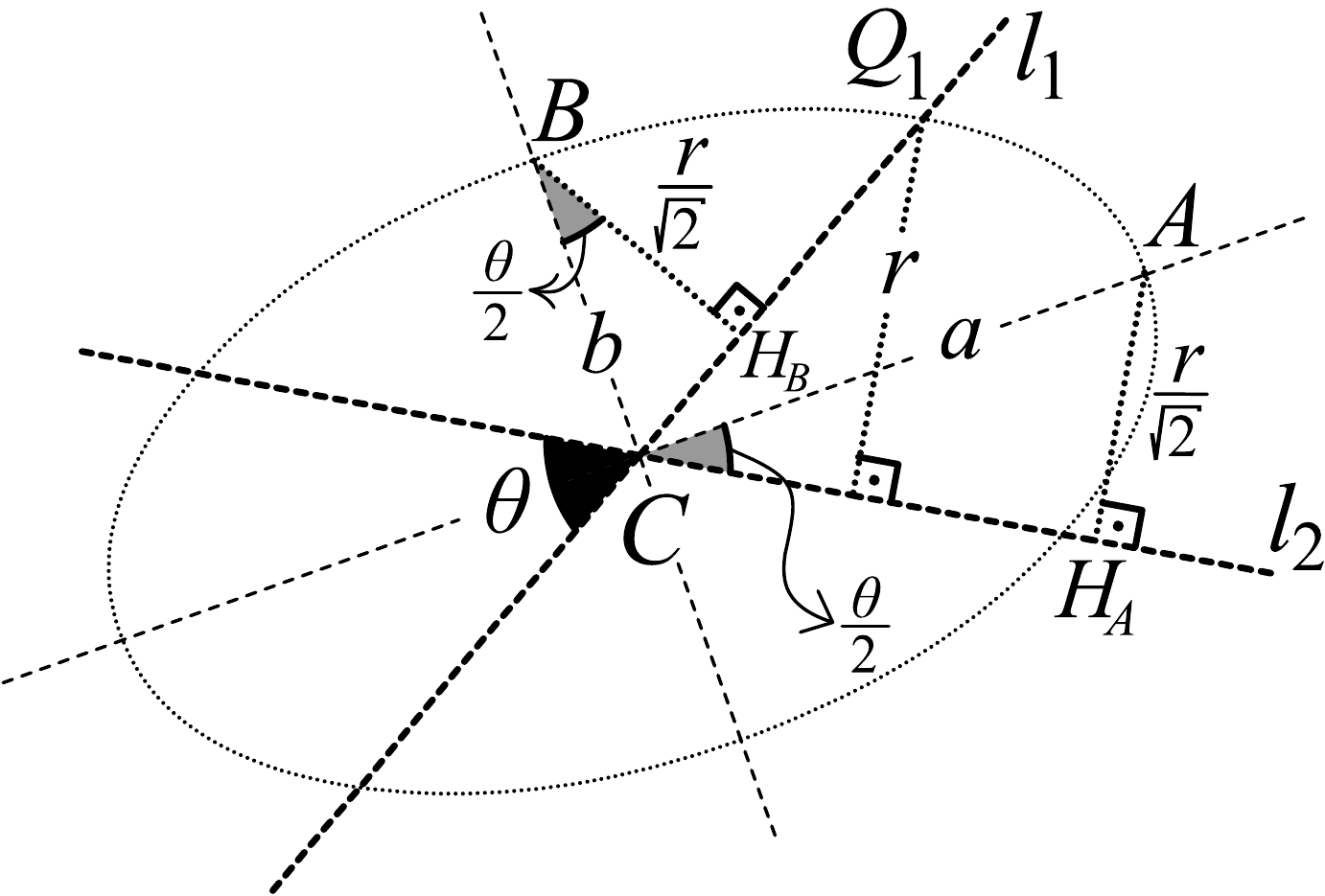} \vspace{-0.2in}
\end{center}
\begin{center}
\textbf{Figure 13.} A $(v_{1},v_{2})$-Euclidean circle and an ellipse that are the same.   
\end{center} \vspace{-0.05in}

\noindent Notice that $r^{2}$ is the harmonic mean of $b^{2}$ and $a^{2}$,
so we have $b\leq r\leq a$ where the equality holds only for the case $a=b=r$. 
Another fact is the chords derived by the lines $l_{1}$ and $l_{2}$ have
the same length, and if $d_{E}(C,Q_{i})=R$ then $\sin \theta =\frac{r}{R}$,
and we get  \vspace{-0.05in}
\begin{equation*}
R=\sqrt{a^{2}+b^{2}}/\sqrt{2}\text{ \ and \ }Rr=ab. \vspace{-0.05in}
\end{equation*}
Since chords derived by the lines $l_{1}$ and $l_{2}$ are conjugate
diameters by the following theorem, the last two equalities can also be
derived by the first and the second theorems of Appollonius; which are 
\newline
(1) The sum of the squares of any two conjugate semi-diameters is equal to $a^{2}+b^{2}$, \newline
(2) The area of the parallelogram determined by two coterminous conjugate
semi-diameters is equal to $ab$ (see \cite[pp. 1800-1803]{Mccartin}). \vspace{-0.05in}

\begin{theorem}
The chords derived $l_{1}$ and $l_{2}$ are conjugate diameters
of the ellipse. \vspace{-0.1in}
\end{theorem}

\begin{proof}
We know that the diameters parallel to any pair of supplemental chords
(which are formed by joining the extremities of any diameter to a point
lying on the ellipse) are conjugate (see \cite[p. 1805]{Mccartin}). Since 
$l_{1}$ and $l_{2}$ are parallel to a pair of supplemental chords formed by
joining the extremities of the minor axis to one of the extremities of the
major axis, the chords derived by the lines $l_{1}$ and $l_{2}$ are
conjugate diameters of the ellipse. \vspace{-0.05in}
\end{proof}

\begin{remark}
Since the chords derived by the lines $l_{1}$ and $l_{2}$ are conjugate, 
$l_{1}$ is parallel to the tangent lines through the extremities of the chord
determined by $l_{2}$, and vice versa. It is clear that the tangent lines
through the extremities of these conjugate diameters determine a rhombus
with sides of length $2R$, circumscribed the ellipse. Since $l_{1}$ and $l_{2}$ 
are symmetric about the axes of the ellipse, the diagonals of the
rhombus are on the axes of the ellipse, and they have lengths $2\sqrt{2}a$
and $2\sqrt{2}b$. Similarly, the chords determined by the axes of the
ellipse are also conjugate, and the tangent lines through the extremities of
these conjugate diameters determine a rectangle with sides of lengths $2a$
and $2b$, circumscribed the ellipse. Since $\tan \frac{\theta }{2}=\frac{b}{a}$, 
the diagonals of this rectangle are on the lines $l_{1}$ and $l_{2}$,
and they have length $2\sqrt{2}R$ (see Figure 14). Notice that there is an
ellipse similar to the prior, through eight vertices of these rectangle and
rhombus, whose semi-major and semi-minor axes are equal to $\sqrt{2}a$ and 
$\sqrt{2}b$ respectively (see Figure 14).\vspace{-0.15in}
\end{remark}

\begin{center}
\includegraphics[width=1.9 in]{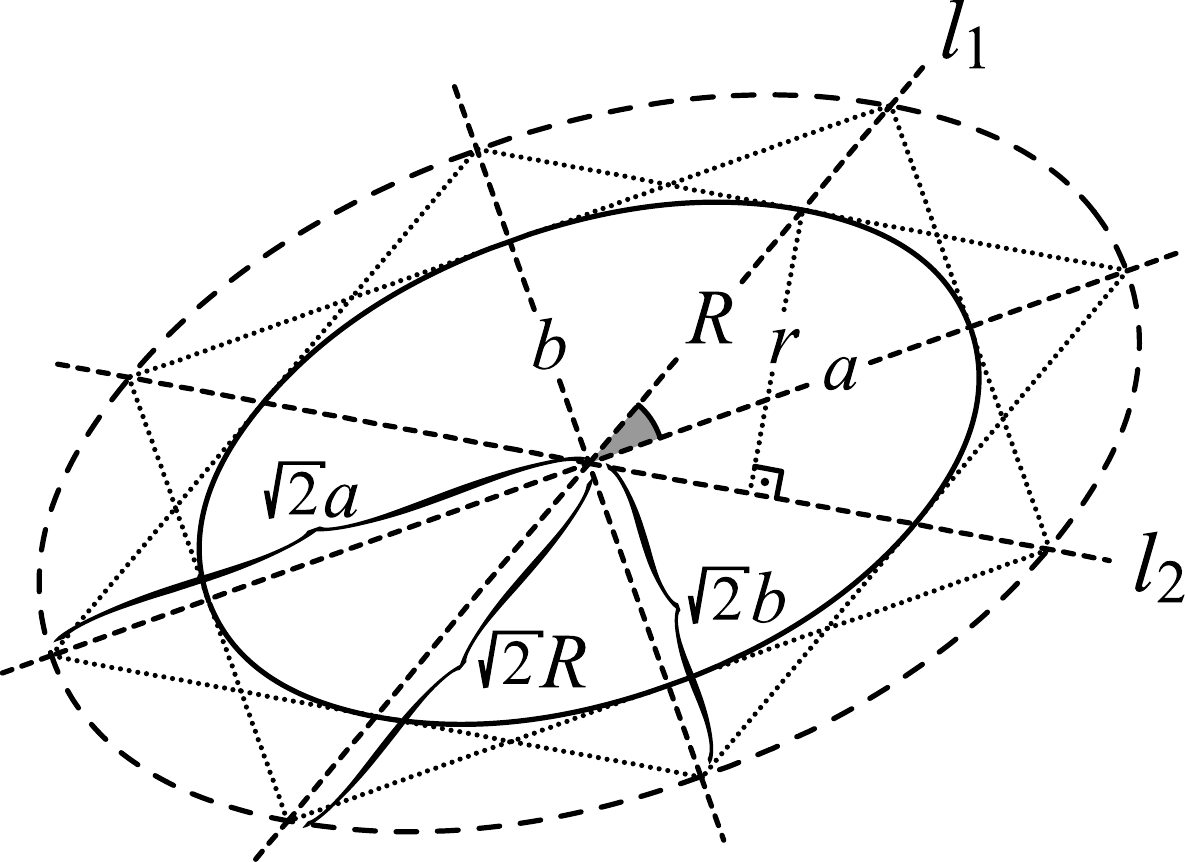} \vspace{-0.15in}
\end{center}
\begin{center}
\textbf{Figure 14.} The rectangle and the rhombus determined by an ellipse. 
\end{center}

\begin{remark}
Observe that, the rhombus derived by the tangent lines through the
extremities of the conjugate diameters determined by the lines $l_{1}$ and 
$l_{2}$, is the $(v_{1},v_{2})$-maximum circle, and the rectangle whose
vertices are the midpoints of this rhombus is the $(v_{1},v_{2})$-taxicab
circle, having the same center and radius of the $(v_{1},v_{2})$-Euclidean
circle, for $\lambda _{1}=\lambda _{2}=1$ and linearly independent unit
vectors $v_{1}$ and $v_{2}$. Figure 15 illustrates $(v_{1},v_{2})$-taxicab, 
$(v_{1},v_{2})$-Euclidean and $(v_{1},v_{2})$-maximum circles with the same
center and radius, for $\lambda _{1}=\lambda _{2}=1$ and linearly
independent unit vectors $v_{1}$and $v_{2}$, each of which is perpendicular
to one of the lines $l_{1}$ and $l_{2}$.\vspace{-0.15in}
\end{remark}

\begin{center}
\includegraphics[width=1.9 in]{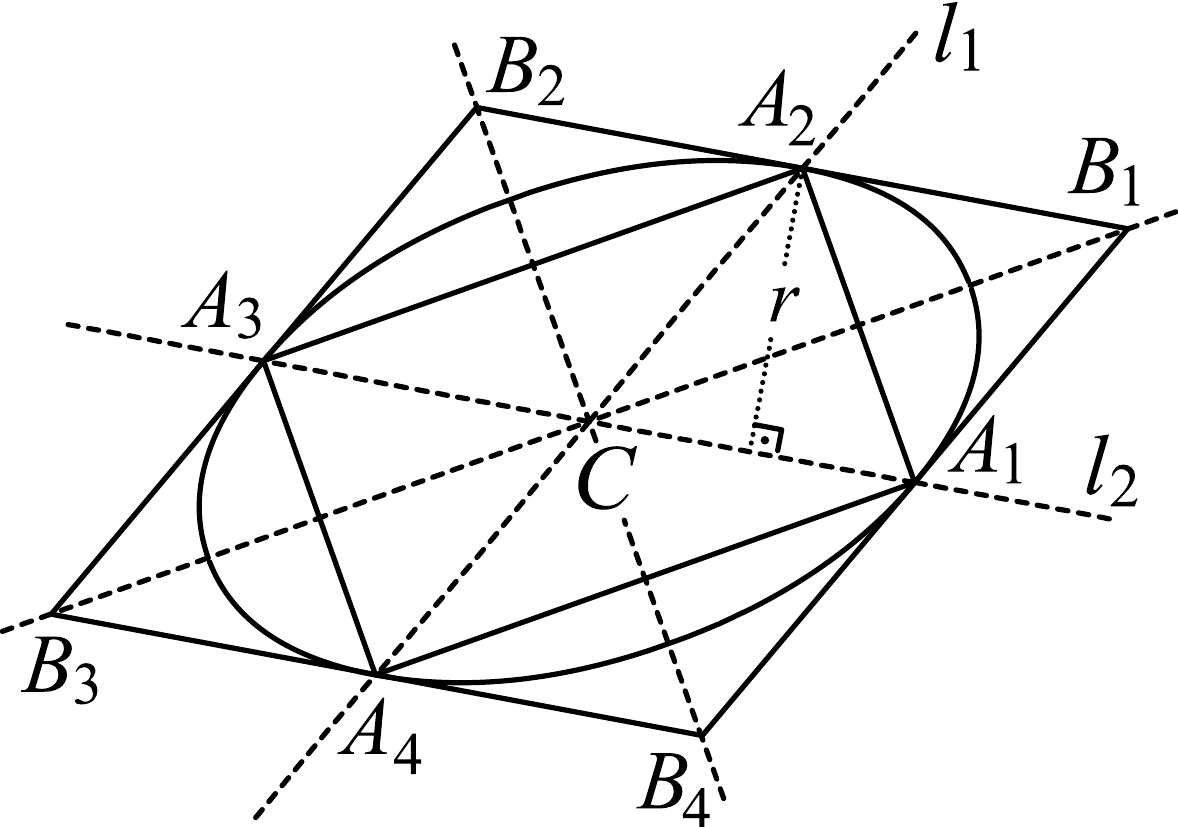} \vspace{-0.17in}
\end{center}
\begin{center}
\textbf{Figure 15.} $(v_{1},v_{2})$-Minkowski circles with the same center and radius. 
\end{center}

\noindent The following theorem shows that every ellipse is a 
$(v_{1},v_{2})$-Euclidean circle with the same center for a proper generalized Euclidean
metric $d_{E(v_{1},v_{2})}$ with $\lambda _{1}=\lambda _{2}=1$:

\begin{theorem}
Every ellipse with semi-major axis $a$ and semi-minor axis $b$, is a 
$(v_{1},v_{2})$-Euclidean circle with the same center and the radius 
$\frac{\sqrt{2}ab}{\sqrt{a^{2}+b^{2}}}$, for $\lambda _{1}=\lambda _{2}=1$ and
linearly independent unit vectors $v_{1}$ and $v_{2}$, each of which is
perpendicular to a diagonal lines of the rectangle circumscribed the
ellipse, whose sides are parallel to the axes of the ellipse.
\end{theorem}

\begin{proof}
Since Euclidean distances are preserved under rigid motions, without loss of
generality, let us consider the ellipse with the equation 
\begin{equation}
\frac{x^{2}}{a^{2}}+\frac{y^{2}}{b^{2}}=1.
\end{equation}
Take the lines $l_{1}:bx-ay=0$ and $l_{2}:bx+ay=0$ passing through the
origin, since they are diagonal lines of the rectangle circumscribed the
ellipse, whose sides are parallel to the axes of the ellipse. So, for every
point $P=(x_{0},y_{0})$ on the ellipse, we have 
\begin{equation}
\lbrack (d_{E}(P,l_{1}))^{2}+(d_{E}(P,l_{2}))^{2}]^{1/2}=\left[ \frac{%
2(bx_{0}^{2}+ay_{0}^{2})}{a^{2}+b^{2}}\right] ^{1/2}=\frac{\sqrt{2}ab}{\sqrt{%
a^{2}+b^{2}}}
\end{equation}
which is a constant. Thus, for $\lambda _{1}=\lambda _{2}=1$ and linearly
independent unit vectors $v_{1}$ and $v_{2}$, each of which is perpendicular
to one of the lines $l_{1}$ and $l_{2}$, the ellipse is a 
$(v_{1},v_{2})$-Euclidean circle with center $O$ and radius $\sqrt{2}ab/\sqrt{a^{2}+b^{2}}.$
\medskip
\end{proof}

\noindent Now, by Theorem 5.1 and Theorem 5.2, for any positive real number $r$ 
and two distinct lines $l_{1}$ and $l_{2}$ intersecting at a point $C$,
every point $P$ satisfying the equation 
\begin{equation}
(d_{E}(P,l_{1}))^{2}+(d_{E}(P,l_{2}))^{2}=r^{2}
\end{equation}
is on an ellipse with the center $C$, semi-major axis $a=\frac{r}{\sqrt{%
1-\cos \theta }}$ and semi-minor axes $b=\frac{r}{\sqrt{1+\cos \theta }}$,
where $\theta $ is the non-obtuse angle between the lines $l_{1}$ and $l_{2}$, 
and the lines $l_{1}$ and $l_{2}$ are diagonal lines of the rectangle
circumscribed the ellipse, whose sides are parallel to the axes of the
ellipse. Conversely, by the Theorem 5.4, any point $P$ on this ellipse
satisfies the equation 
\begin{equation}
(d_{E}(P,l_{1}))^{2}+(d_{E}(P,l_{2}))^{2}=\frac{2\frac{r^{2}}{1-\cos \theta }%
\frac{r^{2}}{1+\cos \theta }}{\frac{r^{2}}{1-\cos \theta }+\frac{r^{2}}{%
1+\cos \theta }}=r^{2}.
\end{equation}
Clearly, for every ellipse, there are unique pair of lines $l_{1}$ and $l_{2}$, and
there is unique constant $r^{2}$ which is the square of the distance from an
intersection point of the ellipse and one of the lines $l_{1}$ and $l_{2}$
to the other one of them. Notice that we discover a new definition of the
ellipse:\medskip

\begin{definition}
In the Euclidean plane, an ellipse is a set of all points for each of which
sum of squares of its distances to \textbf{two intersecting fixed lines} is 
\textbf{constant}. We call each such fixed line an \textbf{eccentrix of the
ellipse}, and call the chord determined by an eccentrix \textbf{eccentric
diameter of the ellipse}, and half of an eccentric diameter \textbf{eccentric radius of the ellipse}.
\end{definition}

\noindent Clearly, eccentrices of an ellipse determine the eccentricity -so,
the shape- of the ellipse, and vice versa, since the eccentricity is 
\begin{equation}
e=\sqrt{1-\tfrac{b^{2}}{a^{2}}}=\sqrt{1-\tan ^{2}\tfrac{\theta }{2}}
\end{equation}%
where $\theta $ is the non-obtuse angle between\ the eccentrices. Notice
that ellipses with the same eccentrices -more generally, ellipses having the
same angle between their eccentrices- are similar, since they have the same
eccentricity (see Figure 16).

\begin{center}
\includegraphics[width=2.1 in]{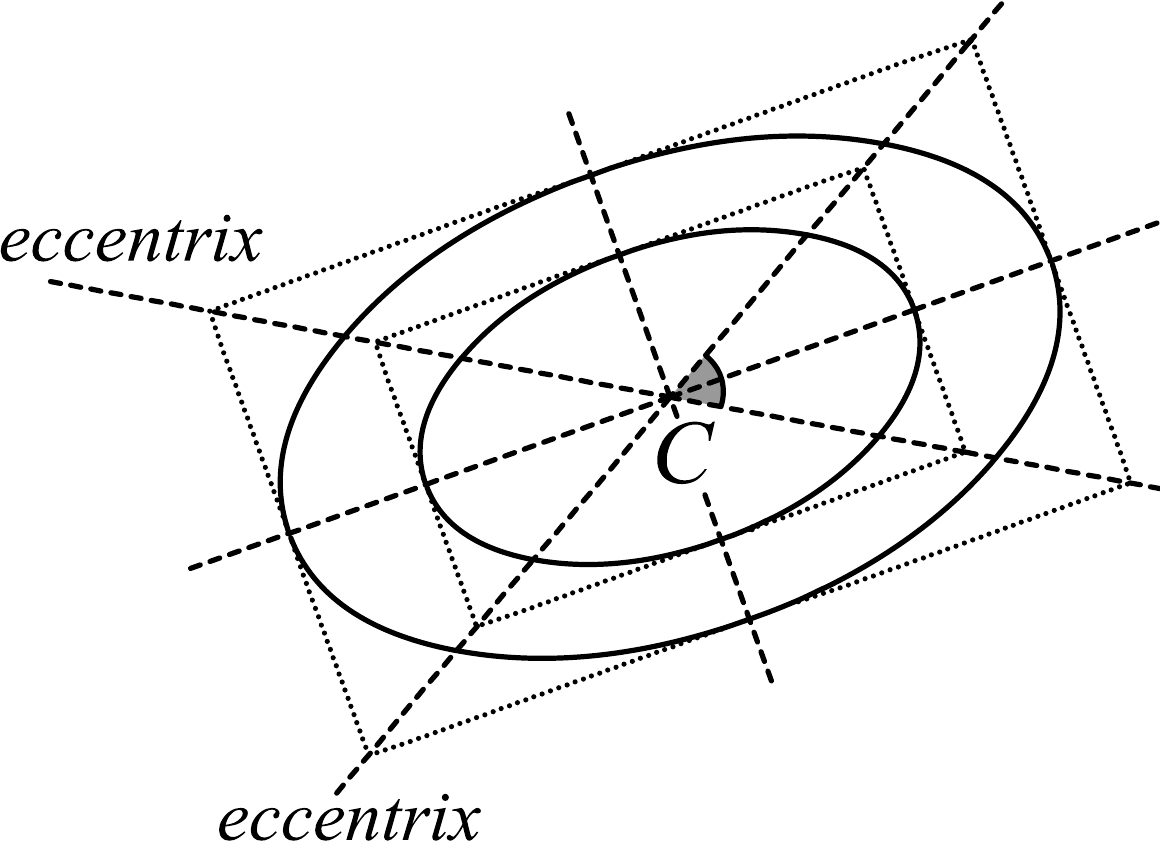}
\end{center}
\begin{center}
\textbf{Figure 16.} Ellipses with the same eccentrices.  
\end{center}

\noindent Related to this new "two-eccentrices" definition of the ellipse,
we immediately have the following fundamental conclusions:

\begin{corollary}
Given a constant $c\in \mathbb{R}^{+}$ and two fixed lines 
$l_{1}$ and $l_{2}$ intersecting at a point $C$,
having the non-obtuse angle $\theta $ between them. Then the ellipse with
constant $c$ and eccentrices $l_{1}$ and $l_{2}$ is the ellipse with the
center $C$, having semi-major axis $a=\frac{\sqrt{c}}{\sqrt{1-\cos \theta }}$
and semi-minor axis $b=\frac{\sqrt{c}}{\sqrt{1+\cos \theta }}$, such that
the major axis of the ellipse is the angle bisector of $\theta $. In
addition, this ellipse is a $(v_{1},v_{2})$-Euclidean circle with respect to
the $(v_{1},v_{2})$-Euclidean metric, for $\lambda _{1}=\lambda _{2}=1$ and
linearly independent unit vectors $v_{1}$ and $v_{2}$, each of which is
perpendicular to one of the eccentrices of the ellipse, having the center $C$
and radius $\sqrt{c}$, which is the Euclidean distance from the intersection
point of an eccentrix and the ellipse to the other eccentrix.
\end{corollary}

\begin{corollary}
Given an ellipse with center $C$, semi-major axis $a$ and\ semi-minor axis 
$b $. Then, the eccentrices of this ellipse are diagonal lines of the
rectangle circumscribed the ellipse whose sides are parallel to the axes of
the ellipse, and the constant of this ellipse\emph{\ }is $\frac{2a^{2}b^{2}}{a^{2}+b^{2}}$, 
which is the square of the Euclidean distance from the
intersection point of an eccentrix and the ellipse to the other eccentrix.
In addition, this ellipse is a $(v_{1},v_{2})$-Euclidean circle with center 
$C$ and radius $\frac{\sqrt{2}ab}{\sqrt{a^{2}+b^{2}}}$, with respect to the 
$(v_{1},v_{2})$-Euclidean metric, for $\lambda _{1}=\lambda _{2}=1$ and
linearly independent unit vectors $v_{1}$ and $v_{2}$, each of which is
perpendicular to one of the eccentrices of the ellipse.
\end{corollary}

\begin{corollary}
A $(v_{1},v_{2})$-Euclidean circle with center $C$ and radius $r$, is an
ellipse whose constant is $r^{2}$ and eccentrices are the lines through $C$
and perpendicular to $v_{1}$ and $v_{2}$. In addition, semi-major and
semi-minor axes of this ellipse are $a=\frac{r}{\sqrt{1-\cos \theta }}$ and 
$b=\frac{r}{\sqrt{1+\cos \theta }}$ where $\cos \theta =\left\vert
v_{11}v_{21}+v_{12}v_{22}\right\vert $.
\end{corollary}

\begin{remark}
Clearly, an ellipse can be determined uniquely by its axes, semi-major axis 
$2a$ and minor axis $2b$, or simply a rectangle with sides of lengths $2a$
and $2b$. Here, we see that it can also be determined uniquely by its
eccentrices with the angle between them and eccentric diameter $2R$, or
simply a rhombus with sides of length $2R$; having the relation $\sqrt{2}R=\sqrt{a^{2}+b^{2}}$. 
While diagonals of the rectangle (whose length is equal
to $2\sqrt{2}R$) give eccentrices of the ellipse, diagonals of the rhombus
(whose lengths are equal to $2\sqrt{2}a$ and $2\sqrt{2}b$) give axes of the
ellipse (see Figure 17). Obviously, when such a rectangle is given, one can
construct the related rhombus, and vice versa. So, for an ellipse whose
center and four points of tangency to its rectangle are known, one can
construct four points on the eccentrices of the ellipse: they are
intersection points of diagonals of the rectangle and sides of the rhombus.
Similarly, for an ellipse whose center and four points of tangency to its
rhombus are known, one can construct four more points on the ellipse: they
are intersection points of diagonals of the rhombus and sides of the
rectangle (see also \cite{Mccartin2} and \cite{Horvath} for construction of
an ellipse from a pair of conjugate diameters).
\end{remark}

\begin{center}
\includegraphics[width=2.5 in]{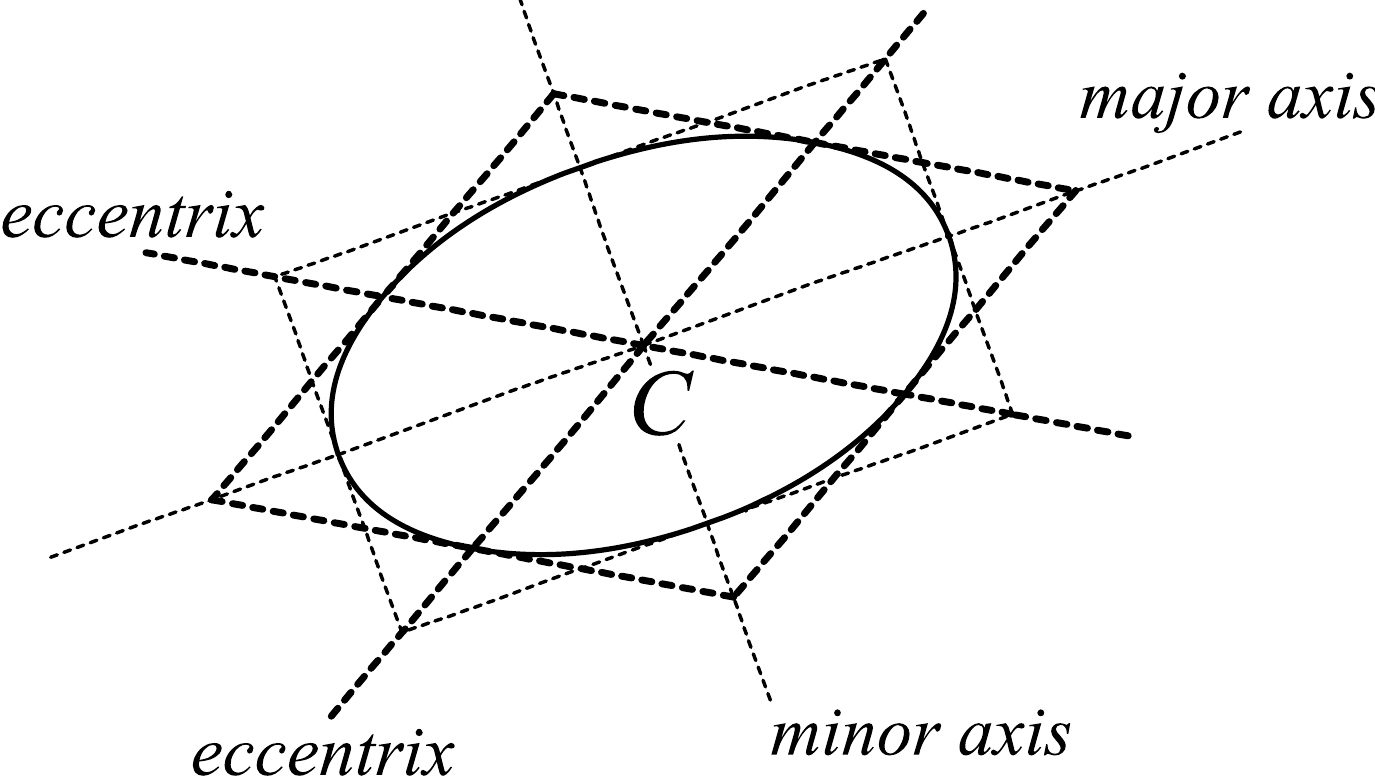}
\end{center}
\begin{center}
\textbf{Figure 17.} Axes and eccentrices of an ellipse.
\end{center}

\begin{remark}
Clearly, the eccentrices of ellipse with the equation%
\begin{equation*}
\frac{x^{2}}{a^{2}}+\frac{y^{2}}{b^{2}}=1
\end{equation*}
and the asymptotes of conjugate hyperbolas with the equations%
\begin{equation*}
\frac{x^{2}}{a^{2}}-\frac{y^{2}}{b^{2}}=1\text{ \ and\ \ }-\frac{x^{2}}{a^{2}}+\frac{y^{2}}{b^{2}}=1
\end{equation*}
are the same (see Figure 18).
\end{remark}

\begin{center}
\includegraphics[width=3.3 in]{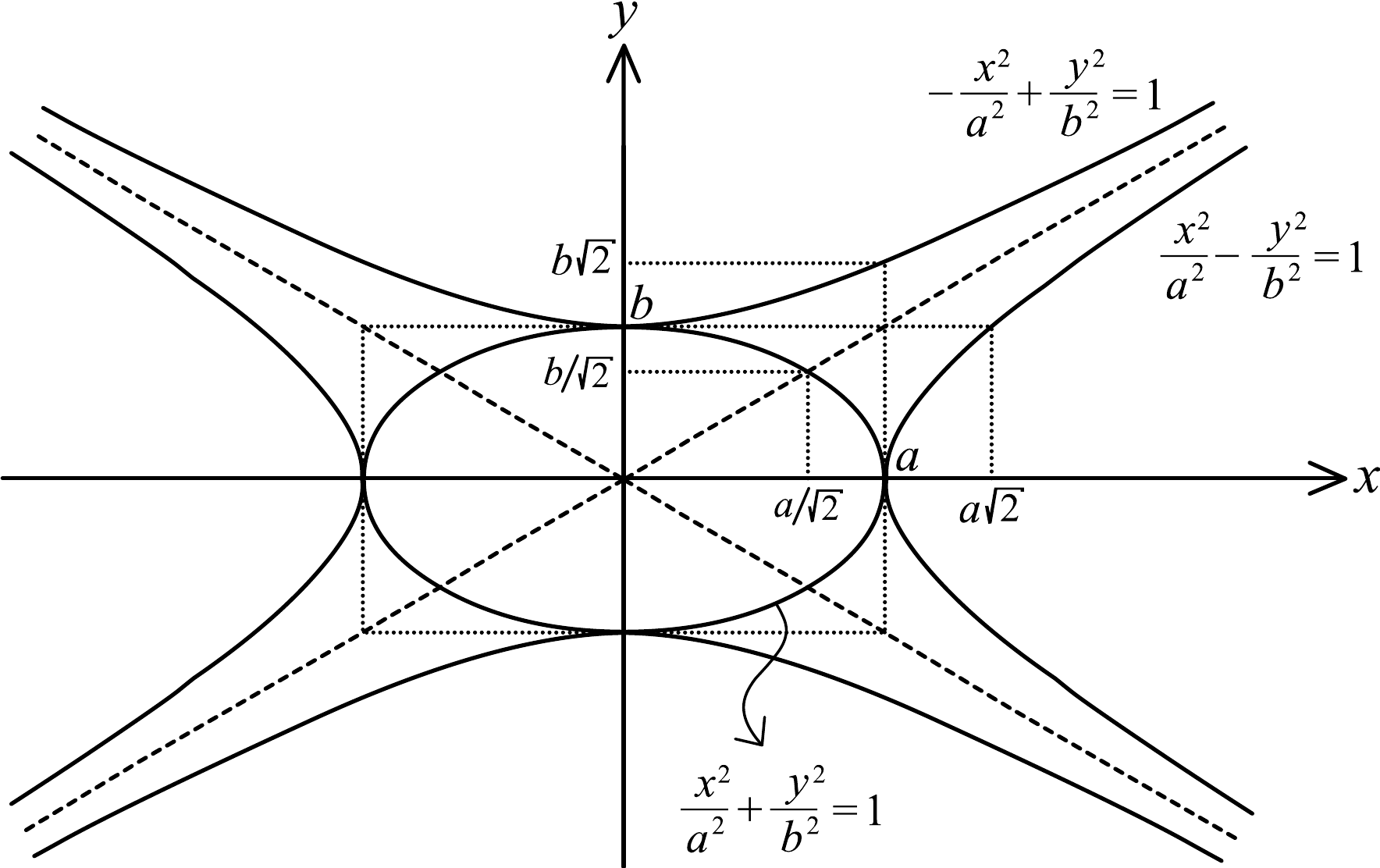} \vspace{-0.06in}
\end{center}
\begin{center}
\textbf{Figure 18.} Conics determined by the same semi-axes $a$ and $b$.
\end{center}

\noindent One can naturally wonder the answer of the following question: So,
what is the set of all points for each of which difference of squares of its
distances to two intersecting fixed lines is constant, and is it a
definition for hyperbola? It can easily be seen that for any two
intersecting fixed lines this set determines a hyperbola having
perpendicular asymptotes: Consider the same lines $l_{i}:v_{i1}x+v_{i2}y=0$
for $i=1,2$, and the set of points satisfying the equation 
\begin{equation}
(d_{E}(P,l_{1}))^{2}-(d_{E}(P,l_{2}))^{2}=k
\end{equation}%
for $k\in \mathbb{R}$-$\{0\}$, which gives the equation
\begin{equation}
Ax^{2}+By^{2}+2Cxy+2Dx+2Ey+F=0
\end{equation}%
where $A=(v_{11}^{2}-v_{21}^{2})$, $B=(v_{12}^{2}-v_{22}^{2})$, 
$C=(v_{11}v_{12}-v_{21}v_{22})$, $D=E=0$ and $F=-k$. Since $\delta <0$ and 
$\Delta \neq 0$, this equation determines a hyperbola 
(see \cite[pp. 232-233]{Zwil}). Moreover, by the theorem given in \cite{Venit}, if $v_{12}\neq
v_{22}$ then slopes $m_{1}$ and $m_{2}$ of the asymptotes are the distinct
real roots of the quadratic equation
\begin{equation}
(v_{12}^{2}-v_{22}^{2})m^{2}+2(v_{11}v_{12}-v_{21}v_{22})m+(v_{11}^{2}-v_{21}^{2})=0%
\text{,}
\end{equation}
so the asymptotes are perpendicular since the multiplication of the roots is
-$1$, and if $v_{12}=v_{22}$ then the hyperbola has perpendicular asymptotes
one of them is vertical the other one is horizontal. On the other hand, if
we use absolute value for the difference notion in the question, then
clearly we get two conjugate hyperbolas having the same perpendicular
asymptotes. So, this set does not give a definition for hyperbola 
(see \cite{Mitchell} for the hyperbolas determined by two fixed lines using the
distances of a point to the fixed lines). \smallskip

Notice that, in the section 3 and 4 one can give definitions for rectangle
and rhombus using two distinct lines and determine some properties of them,
in a similar way.

\end{document}